\documentclass[12pt]{article}
\usepackage{graphicx}
\usepackage{amsmath}
\usepackage{amssymb}
\usepackage{theorem}

\numberwithin{equation}{section}

\newtheorem{thm}{Theorem}[section]
\newtheorem{lemma}[thm]{Lemma}
\newtheorem{prop}[thm]{Proposition}
\newtheorem{cor}[thm]{Corollary}
{\theorembodyfont{\rmfamily}
\newtheorem{example}[thm]{Example}
\newtheorem{rmk}[thm]{Remark}
}

 \usepackage[colorlinks=true, linkcolor=blue, urlcolor=red, citecolor=blue]{hyperref}

\newcommand{\qed}{\hfill \mbox{\raggedright \rule{.07in}{.1in}}}
 
\newenvironment{proof}{\vspace{1ex}\noindent{\bf
Proof}\hspace{0.5em}}{\hfill\qed\vspace{1ex}}
\newenvironment{pfof}[1]{\vspace{1ex}\noindent{\bf Proof of
#1}\hspace{0.5em}}{\hfill\qed\vspace{1ex}}

\newcommand{\R}{{\mathbb R}}
\newcommand{\Z}{{\mathbb Z}}
\newcommand{\N}{{\mathbb N}}
\newcommand{\E}{{\mathbb E}}
\newcommand{\PP}{{\mathbb P}}

\newcommand{\cB}{{\mathcal B}}
\newcommand{\cC}{{\mathcal C}}
\newcommand{\cJ}{{\mathcal J}}
\newcommand{\cM}{{\mathcal M}}
\newcommand{\cW}{{\mathcal W}}

\newcommand{\tH}{{\widetilde H}}
\newcommand{\tN}{{\widetilde N}}
\newcommand{\tv}{{\widetilde v}}
\newcommand{\tW}{{\widetilde W}}
\newcommand{\tY}{{\widetilde Y}}

\newcommand{\tvarphi}{{\widetilde \varphi}}

\newcommand{\hf}{{\hat f}}
\newcommand{\hM}{{\widehat M}}
\newcommand{\hS}{{\widehat S}}
\newcommand{\hV}{{\widehat V}}
\newcommand{\hphi}{{\hat \phi}}
\newcommand{\hPhi}{{\widehat \Phi}}
\newcommand{\hpsi}{{\widehat \psi}}
\newcommand{\hPsi}{{\widehat \Psi}}
\newcommand{\eps}{\epsilon}
\newcommand{\sgn}{\operatorname{sgn}}
\newcommand{\SMALL}{\textstyle}

\title{Convergence to a L\'evy process in the Skorohod $\cM_1$ and $\cM_2$ topologies for nonuniformly hyperbolic systems,
 including billiards with cusps}

 \author{
 Ian Melbourne\thanks{Mathematics Institute, University of Warwick, Coventry, CV4 7AL, UK.  \newline i.melbourne@warwick.ac.uk}
\and
Paulo Varandas\thanks{Departamento de Matem\'atica, Universidade Federal da Bahia,
40170-110 Salvador, Brazil.  \newline pcvarand@gmail.com
}}

\date{21 September 2018.  Updated 15 May 2019}

\begin{document}

 \maketitle

 \begin{abstract}
We prove convergence to a L\'evy process for a class of dispersing billiards with cusps.  For such examples, convergence to a stable law was proved by Jung \& Zhang.  For the corresponding functional limit law, convergence is not possible in the usual Skorohod $\cJ_1$ topology.
Our main results yield elementary geometric conditions for convergence (i) in $\cM_1$, (ii) in $\cM_2$ but not $\cM_1$.

In general, we show 
for a large class of nonuniformly hyperbolic systems how to deduce functional limit laws once convergence to the corresponding stable law is known.
 \end{abstract}

\section{Introduction} 
\label{sec:intro}

It is by now well-known that deterministic dynamical systems often satisfy statistical limit theorems from classical probability theory.
Following Sinai~\cite{Sinai70}, 
a rich source of examples is provided by dispersing billiards~\cite{ChernovMarkarian} which are based on deterministic Lorentz gas models~\cite{Lorentz05}.
By~\cite{BunimSinai81,BunimovichSinaiChernov91},  
the central limit theorem (CLT) and functional central limit theorem or weak invariance principle (WIP) hold
for planar periodic dispersing billiards.
The CLT asserts convergence to a normal distribution and the WIP deals with convergence to the corresponding Brownian notion.
These limit laws also hold for Sinai billiards where the boundary of the table is a simple closed curve consisting of finitely many $C^3$ convex inwards curves with nonvanishing curvature and nonzero angles at corner points~\cite{SimoiToth14}.
For billiards with cusps (corner points with zero angle), the CLT and WIP were obtained by~\cite{BalintChernovDolgopyat11} but with the weakly superdiffusive normalization $(n\log n)^{1/2}$ instead of the standard diffusion rate $n^{1/2}$.  

Stronger superdiffusion rates $n^{1/\alpha}$, $\alpha<2$, with limiting fluctuations governed by an
$\alpha$-stable L\'evy process rather than a Brownian motion,
 have been the focus of much attention across the physical sciences.   See for example~\cite{BL,GaspardWang88, Gouezel04, KGSSR, MZ15, Metzler, PVHS, SamorodnitskyTaqqu, Swinney, Whitt} and references therein.
Whereas Brownian motions are continuous processes with finite variance, L\'evy processes  exhibit jumps of all sizes and have infinite variance. 

In this paper, we show for the first time that convergence to a L\'evy process occurs in dispersing billiards.  The example is elementary to write down and the mechanism for superdiffusion is intuitively transparent.  Moreover, our analysis casts light on the mode of convergence, an aspect which has received little attention previously.

Recently, Jung \& Zhang~\cite{JungZhangsub} considered a class of billiards with cusps where there is vanishing curvature at the cusp and proved convergence to totally skewed $\alpha$-stable laws with $\alpha\in(1,2)$.  However, they were unable to prove the functional WIP version of their limit law (i.e.\ weak convergence to the corresponding
$\alpha$-stable L\'evy process).

In this paper, as part of a general framework including~\cite{JungZhangsub}, we show how to pass from the stable law to the WIP.
The standard $\cJ_1$ Skorohod topology~\cite{Skorohod56,Whitt} is always too strong for these examples, but we obtain convergence in the $\cM_1$ and $\cM_2$ topologies.
The definition of these Skorokhod topologies is recalled in Appendix~\ref{app:topologies}.

It is well-known that the $\cJ_1$ topology is often too strong, and there are many
natural examples where the $\cM_1$ topology is the appropriate one, see for example~\cite{AvramTaqqu92,BasrakKrizmanicSegers12,MZ15,Whitt}.  Indeed,
Whitt~\cite[p.~xii]{Whitt} writes
{\advance\leftmargini -1.3em
\begin{quote}
Thus, while the $\cJ_1$ topology sometimes cannot be used, the
$\cM_1$ topology can almost always be used. Moreover, the extra strength of the $\cJ_1$
topology is rarely exploited. Thus, we would be so bold as to suggest that, \emph{if only
one topology on the function space D is to be considered, then it should be the $\cM_1$
topology}.
\end{quote}
}
\noindent Jakubowski~\cite{Jakubowski07} writes
{\advance\leftmargini -1.3em
\begin{quote}
All these reasons bring interest also to the
weaker Skorokhod's topologies $\cJ_2$, $\cM_1$ and $\cM_2$. Among them practically only
the topology $\cM_1$ proved to be useful.
\end{quote}
}

\noindent Nevertheless, in this paper we provide natural examples where the $\cM_1$ topology is too strong and the $\cM_2$ topology is the appropriate one.
The only previous such example that we know of can be found in~\cite{BasrakKrizmanic14}.

\begin{example} \label{ex:JZ}
We consider the Jung \& Zhang example~\cite{JungZhangsub} consisting of a planar dispersing billiard with a cusp at a flat point.  
A standard reference for background material on billiards is~\cite{ChernovMarkarian}.

The billiard table $Q\subset\R^2$ has a boundary consisting of a finite number of $C^3$ curves $\Gamma_i$, $i=1,\dots,n_0$, where $n_0\ge3$ with a cusp formed by two of these curves $\Gamma_1$, $\Gamma_2$.
In coordinates $(s,z)\in\R^2$, the cusp lies at $(0,0)$ and $\Gamma_1$, $\Gamma_2$ are tangent to the $s$-axis at $(0,0)$.  Moreover, close to $(0,0)$, 
we have $\Gamma_1=\{(s,\beta^{-1}s^\beta)\}$,
$\Gamma_2=\{(s,-\beta^{-1}s^\beta)\}$, where $\beta>2$.
See Figure~\ref{fig:JZ}.
\footnote{In~\cite{JungZhangsub}, it is assumed in addition that 
the trajectory running out of the cusp along the $s$-axis hits $\Gamma_3$ perpendicularly, but this was only done for convenience and is not present in~\cite{JungPeneZhangsub}.}

\begin{figure}[htb]
\centerline{
\includegraphics[scale=.5]{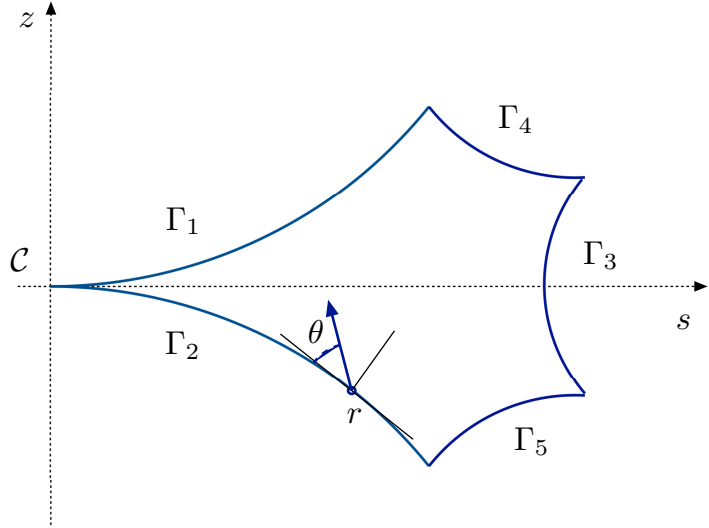}}
\caption{Billiard with a cusp at a flat point as studied by Jung \& Zhang.}
\label{fig:JZ}
\end{figure}

The phase space of the billiard map (or collision map) $T$ is given by $\Lambda=\partial Q\times[0,\pi]$, with coordinates $(r,\theta)$ where $r$ denotes arc length along $\partial Q$ and $\theta$ is the angle between the tangent line of the boundary and the collision vector in the clockwise direction.
There is a natural ergodic invariant probability measure $d\mu=(2|\partial Q|)^{-1}\sin\theta\,dr\,d\theta$ on~$\Lambda$, where $|\partial Q|$ is the length of $\partial Q$.

In configuration space,
the cusp is a single point $(0,0)=\Gamma_1\cap\Gamma_2$.
Let $r'\in\Gamma_1$ and
$r''\in\Gamma_2$ be the arc length coordinates of $(0,0)$.
Then in phase space $\Lambda$, the cusp is
the union of two line segments 
\[
\cC=\{(r',\theta):0\le\theta\le\pi\} \cup\{(r'',\theta):0\le\theta\le\pi\}.
\]

Let $v:\Lambda\to\R$ be a H\"older continuous observable with $\int_\Lambda v\,d\mu=0$ and define\footnote{Our definitions differ from those in~\cite{JungZhangsub} by constant factors, leading to simpler formulas in Section~\ref{sec:JZ}.}
\begin{align} \label{eq:Iv}
I_v(s)=\frac12\int_0^s \{v(r',\theta)+v(r'',\pi-\theta)\} (\sin\theta)^{1/\alpha}\,d\theta,
\quad  s\in[0,\pi].
\end{align}
where $\alpha=\frac{\beta}{\beta-1}\in(1,2)$.
Suppose that $I_v(\pi)>0$
(the case $I_v(\pi)<0$ is identical with the obvious modifications).
Let $G$ be the totally skewed $\alpha$-stable law with characteristic function
\[
\E(e^{iuG})=\exp\{-|u|^\alpha\sigma^\alpha(1-i\sgn u\tan{\SMALL\frac{\pi\alpha}{2}})\},
\quad \sigma^\alpha=(\beta|\partial Q|2^{\alpha-1})^{-1}I_v(\pi)^\alpha\Gamma(1-\alpha)\cos{\SMALL\frac{\pi\alpha}{2}}.
\]

Jung \& Zhang~\cite[Theorem~1.1]{JungZhangsub} prove:
\begin{thm} \label{thm:JZ}
$n^{-1/\alpha}\sum_{j=0}^{n-1}v\circ T^j\to_d G$. \qed
\end{thm}

Let $D[0,\infty)$ denote the set of real-valued c\`adl\`ag functions (right-continuous with left-hand limits) on $[0,\infty)$,
and let $W\in D[0,\infty)$ be the $\alpha$-stable L\'evy process with $W(1)=_d G$.
Define 
\[
W_n:\Lambda\to D[0,\infty), \qquad  \SMALL W_n(t)=n^{-1/\alpha}\sum_{j=0}^{[nt]-1}v\circ T^j.
\]
Since the increments of $W_n$ are bounded by $n^{-1/\alpha}|v|_\infty$ and $W$ has jumps with probability one, $W_n$ does not converge to $W$ in the $\cJ_1$ topology.  However, the weaker $\cM_1$ topology allows an amalgamation of numerous small increments for $W_n$ to approximate a single jump for $W$.  This is analogous to the situation for intermittent maps of Pomeau-Manneville type~\cite{PomeauManneville80} studied in~\cite{MZ15}.  
In contrast to~\cite{MZ15},
convergence in $\cM_1$ is not automatic.
Instead, there is a simple geometric condition on $v|_\cC$ which characterizes convergence in $\cM_1$:

\begin{thm} \label{thm:JZus1}
$W_n\to_w W$ in $(D[0,\infty),\cM_1)$ if and only if 
$v(r',\theta)+v(r'',\pi-\theta)\ge0$ for all $\theta\in[0,\pi]$.
(Equivalently, $s\mapsto I_v(s)$ is nondecreasing on $[0,\pi]$.)
\end{thm}

We also have a sufficient condition for convergence in the even weaker $\cM_2$ topology.  

\begin{thm} \label{thm:JZus2}
If $I_v(s)\in[0,I_v(\pi)]$ for all $s\in[0,\pi]$, then
$W_n\to_w W$ in $(D[0,\infty),\cM_2)$.
\end{thm}

It is now easy to construct a H\"older continuous mean zero observable $v: \Lambda \to \R$ so that convergence holds in $\cM_2$ but not in $\cM_1$.
For example, choose $v$ so that $v(r',\theta)+v(r'',\pi-\theta)$ 
is positive on $[0,\frac{\pi}{3})\cup(\frac{2\pi}{3},\pi]$ and negative on $(\frac{\pi}{3},\frac{2\pi}{3})$.
See Figure~\ref{fig:I}(b).
The change of sign violates the condition for $\cM_1$-convergence in 
Theorem~\ref{thm:JZus1}, while it is clear that if $v$ is small enough on 
$(\frac{\pi}{3},\frac{2\pi}{3})$
comparable to its values on $[0,\frac{\pi}{3})\cup(\frac{2\pi}{3},\pi]$, then the 
condition for $\cM_2$-convergence in Theorem~\ref{thm:JZus2} is satisfied.

\begin{figure}[htb]
\centerline{
\includegraphics[scale=.7]{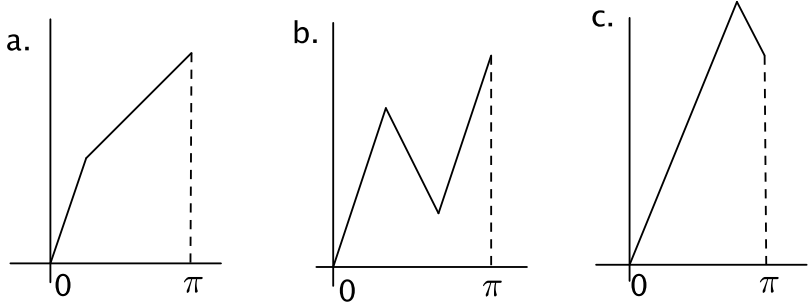}}
\caption{Different possible shapes of the function $I_v$ for
the Jung \& Zhang example:
(a)  WIP holds in the $\cM_1$ (hence also in the $\cM_2$) topology; 
(b) WIP holds in the $\cM_2$ topology but not in the $\cM_1$ topology;
(c) the WIP does not hold even in the $\cM_2$ topology.}
\label{fig:I}
\end{figure}
\end{example}

\begin{rmk}
(a) 
After writing this paper, we learned of independent work of~\cite{JungPeneZhangsub} on billiards with several cusps at flat points.  They considered the case where $v$ has constant sign near each cusp and proved convergence to a L\'evy process in the $\cM_1$ topology.
\\[.75ex]
(b) In a previous version of this paper, we conjectured that the condition in Theorem~\ref{thm:JZus2} for convergence in the $\cM_2$ topology is necessary and sufficient.  This has now been shown to be the case in~\cite{JMPVZ}.  An interesting open question is to consider alternative weaker modes of convergence in situations such as Figure~\ref{fig:I}(c) where $\cM_2$-convergence fails.
(Such a weakening entails diminishing the class of continuous functionals under which weak convergence is preserved.  For example, weak convergence in any of the Skorokhod topologies mentioned above implies weak convergence of the supremum process, i.e.\ $\sup_{[0,t]}W_n\to_w\sup_{[0,t]}W$, see~\cite[Section~13.4]{Whitt}, but this
appears unlikely in the situation of Figure~\ref{fig:I}(c).)
\end{rmk}

\paragraph{Strategy of proof}
The proof of Theorems~\ref{thm:JZus1} and~\ref{thm:JZus2} fits into a general framework~\cite{ChernovZhang05,Markarian04} which has been used to study large classes of examples from billiards specifically and nonuniformly hyperbolic dynamical systems in general.    This framework is described in Section~\ref{sec:prel}.
(It includes the setting of intermittent maps as a very special case, see Remark~\ref{rmk:PM}.)
Let $X\subset\Lambda$ be
a cross-section with first return time 
$\varphi:X\to\Z^+$ and first return map $f=T^\varphi:X\to X$ as in~\eqref{eq:first}.
In Example~\ref{ex:JZ},
$X=(\Gamma_3\cup\dots\cup \Gamma_{n_0})\times[0,\pi]$.   
We require that $f$ is modelled by a Young tower with exponential tails~\cite{Young98} over a ``uniformly hyperbolic'' subset $Y\subset X\subset\Lambda$.  
Associated to the observable $v:\Lambda\to\R$, we have the induced observable
$V=\sum_{\ell=0}^{\varphi-1}v\circ T^\ell:X\to\R$.
Also, associated to $\varphi$, $V$ on $X$ there are induced versions $\varphi^Y$, $V^Y$ on $Y$.

The key argument of~\cite[Theorem~3.1]{JungZhangsub} proves a stable law for
$\varphi:X\to\Z^+$.   Our approach deduces the WIP for $v$ on $\Lambda$ from the stable law for $\varphi$ on $X$.  The idea is to first induce the stable law for $\varphi$ to a stable law for $\varphi^Y$ on $Y$.  Since the dynamics on $Y$ is very well-understood, this leads via results of Gou\"ezel~\cite{Gouezel10b} and Tyran-Kami{\'n}ska~\cite{TyranKaminska10} to convergence to a L\'evy process in the $\cJ_1$ topology for $\varphi^Y$ and thereby $V^Y$.
The WIP for $V^Y$ uninduces to convergence in the $\cM_1$ topology for $V$ on $X$.  Under certain conditions, this uninduces to convergence in the $\cM_1$ or $\cM_2$ topology for $v$.  The strategy can be represented diagrammatically as follows:

\begin{table}[htb]
\centering
\begin{tabular}{ccccc}
$\Lambda$ && $X$ && $Y$  \\[.75ex] \hline \\
    && stable law for $\varphi$ & $\Longrightarrow$ & stable law for $\varphi^Y$  \\
&&&& $\Downarrow$ \\
    &&&& WIP in $\cJ_1$ for $\varphi^Y$  \\
&&&& $\Downarrow$ \\
     WIP in $\cM_1$ / $\cM_2$ for $v$ & $\Longleftarrow$ & 
WIP in $\cM_1$ for $V$  
& $\Longleftarrow$ & WIP in $\cJ_1$ for $V^Y$   
\\ \hline
\end{tabular}
\end{table}

The remainder of the paper is organized as follows.  In Section~\ref{sec:prel}, we consider the Chernov-Markarian-Zhang framework where the underlying system has a first return map modelled by a Young tower with exponential tails.
In Section~\ref{sec:results}, we state our main results on stable laws and WIPs for systems with a Chernov-Markarian-Zhang structure.
In Section~\ref{sec:Minduce}, we state and prove a purely probabilistic result on uninducing WIPs in the $\cM_1$ or $\cM_2$ topology, extending a result of~\cite{MZ15}.
Section~\ref{sec:An} contains limit laws for the return times $\varphi$ and
$\varphi^Y$, and Section~\ref{sec:moment} contains some estimates for induced H\"older observables.  These are combined in Section~\ref{sec:proof} to prove our main results from Section~\ref{sec:results}.
In Section~\ref{sec:JZ}, we return to Example~\ref{ex:JZ}, proving Theorems~\ref{thm:JZus1} and~\ref{thm:JZus2} as well as giving a streamlined proof of Theorem~\ref{thm:JZ}.

\paragraph{Notation}
We use the ``big $O$'' and $\ll$ notation interchangeably, writing $a_n=O(b_n)$ or $a_n\ll b_n$ if there is a constant $C>0$ such that
$a_n\le Cb_n$ for all $n\ge1$.  Also,
we write $a_n\approx b_n$ if $a_n\ll b_n\ll a_n$.
As usual, $a_n\sim b_n$ as $n\to\infty$ means that $\lim_{n\to\infty}a_n/b_n=1$.

For $a,b\in\R$, we write $a\wedge b=\min\{a,b\}$ and
$a\vee b=\max\{a,b\}$.

Recall that a sequence $b_n\in(0,\infty)$ is \emph{regularly varying of index 
$p>0$} if $b_{\lambda n}/b_n\to\lambda^p$ as $n\to\infty$ for all $\lambda\ge1$. 

\section{Preliminaries}
\label{sec:prel}

In this section, we recall the
Chernov-Markarian-Zhang framework~\cite{ChernovZhang05,Markarian04}.
Roughly speaking, this means that there is a convenient first return map that is modelled by
a Young tower with exponential tails~\cite{Young98}.
The full details from Young~\cite{Young98} are not required for our main theorems, so we recall here only those aspects that are needed.

\subsection{Towers and return maps}
\label{sec:tower}

In this subsection, we review a purely measure-theoretic framework of tower maps and return maps that arises throughout this paper.

Let $F:Y\to  Y$ be a measure-preserving transformation on a probability space
$(Y,\mu_Y)$, and let $\tau:Y\to\Z^+$ be integrable.
The tower $\Delta=Y^\tau$ and tower map $\hf:\Delta\to\Delta$ are given by
\begin{align} \label{eq:Delta}
\Delta=\{(y,\ell)\in Y\times\Z:0\le\ell< \tau(y)\}, \quad
\hf(y,\ell)=\begin{cases} 
(y,\ell+1) & \ell\le  \tau(y)-2 \\
(Fy,0) & \ell= \tau(y)-1 \end{cases}.
\end{align}
Define $\bar\tau=\int_Y\tau\,d\mu_Y$.  Then $\mu_\Delta=(\mu_Y\times{\rm counting})/\bar\tau$ is an $\hf$-invariant probability measure on $\Delta$.
We call $\hf:\Delta\to\Delta$ the {\em tower with base map $F$ and return time $\tau$}.

Next, let $f:X\to X$ be a measure-preserving transformation
on a probability space $(X,\mu_X)$,
and $Y\subset X$ a positive measure subset.
Let $\tau:Y\to\Z^+$ be measurable such that
$f^{\tau(y)}y\in Y$ for a.e.\ $y\in Y$; define $F=f^\tau:Y\to Y$.
Suppose that $\mu_Y$ is an $F$-invariant probability measure on $Y$ and that
$\tau$ is integrable with respect to $\mu_Y$.
Let $\hf:\Delta\to\Delta$ denote the tower with base map $F$ and return time $\tau$, and let $\pi:\Delta\to X$ be the semiconjugacy $\pi(y,\ell)=f^\ell y$.  Assume that $\mu_X=\pi_*\mu_\Delta$.
If all these assumptions are satisfied,
we call $\tau$ a {\em return time} and $F$ a {\em return map}.

\subsection{Young towers with exponential tails}
\label{sec:YT}

Let $f:X\to X$ be a measure-preserving transformation defined on a metric 
space $(X,d)$ with Borel probability measure $\mu_X$.   
Suppose that $Y$ is a positive measure subset of $X$ and that
$\tau:Y\to\Z^+$ is a return time with return map
$F=f^\tau:Y\to Y$.  
In particular, there is an $F$-invariant probability measure $\mu_Y$ on $Y$ such 
that $\tau$ is $\mu_Y$-integrable.
Let $\Delta=Y^\tau$ and $\hf:\Delta\to\Delta$ be the tower with base map $F$ and return time $\tau$ as in Subsection~\ref{sec:tower}
with $\hf$-invariant probability measure $\mu_\Delta$ and semiconjugacy $\pi:\Delta\to X$ such that $\mu_X=\pi_*\mu_\Delta$.
In addition, we assume that $\mu_Y$ and $\mu_\Delta$ (and hence $\mu_X$) are ergodic.
Moreover, we assume the exponential tails condition 
\begin{align}\label{eq:exptails}
\mu_Y(y\in Y:\tau(y)>n)=O(e^{-cn})  \quad\text{for some $c>0$}.
\end{align} 

Let $\cW^s$ be a cover of $Y$ by disjoint measurable subsets (called ``local stable leaves'') and
let $W^s_y$ denote the local stable leaf containing $y$.  
We require that $F(W^s_y)\subset W^s_{Fy}$ for all $y\in Y$.  
Let $\bar Y$ be the quotient space obtained from $Y$ by quotienting along local stable manifolds
and denote by $\bar\pi : Y \to \bar Y$ the corresponding projection.
The probability measure $\bar\mu_Y=\bar\pi_*\mu_Y$
is ergodic and invariant under the quotient map $\bar F:\bar Y\to\bar Y$, and
$\bar\pi$ defines a
measure-preserving semiconjugacy between $F$ and $\bar F$.

Let $\{a\}$ be an at most countable measurable partition of $\bar Y$.
Define $s(y,y')$ to be the least integer $n\ge0$ such that $F^ny$, $F^ny'$ lie in distinct partition elements.  It is assumed that $s(y,y')=\infty$ if and only if $y=y'$.  We require that $\bar F|_a:a\to\bar Y$ is a measurable bijection for all $a$ and that there are constants
$C>0$, $\theta\in(0,1)$ such that
\[
\SMALL |\log \frac{d\bar\mu_Y}{d\bar\mu_Y \circ \bar F}(y)-
\log \frac{d\bar\mu_Y}{d\bar\mu_Y \circ \bar F}(y')|\le C\theta^{s(y,y')}
\quad\text{for all $y,y'\in a$ and all $a$.}
\]
Under these conditions,
$\bar F:\bar Y\to\bar Y$ is called a {\em (full branch) Gibbs-Markov map}~\cite{AaronsonDenker01}.

We require that $\tau:Y\to\Z^+$ is constant on $\bar\pi^{-1}a$ for all $a$.
Hence 
$\tau$ is well-defined on $\bar Y$ and constant on partition elements.

Finally, assume that there are constants $C>0$, $\gamma_0\in(0,1)$ such that
\begin{align} \label{eq:Ws}
d(F^ny,F^ny')\le C\gamma_0^n
\quad\text{for all $y,y'\in Y$, $y'\in W^s_y$, $n\ge0$.}
\end{align}

Under these assumptions, we say that $f:X\to X$ is modelled by a {\em Young tower $\Delta=Y^\tau$ with exponential tails}.

\subsection{Chernov-Markarian-Zhang framework}
\label{sec:CMZ}

Let $T:\Lambda\to \Lambda$ be an ergodic measure-preserving transformation defined on a metric 
space $(\Lambda,d)$ with Borel probability measure $\mu$.   
Let $X\subset \Lambda$ be a Borel subset of positive measure and
define the {\em first return time}
$\varphi:X\to\Z^+$ and {\em first return map} $f=T^\varphi:X\to X$,
\begin{align} \label{eq:first}
\varphi(x)=\inf\{n\ge1:T^nx\in X\},
\qquad
f(x)=T^{\varphi(x)}x.
\end{align}
Then $\varphi$ is integrable and
$\mu_X=\mu|_X/\mu(X)$ is an ergodic $f$-invariant probability measure on $X$.
Define $\bar \varphi=\int_X\varphi\,d\mu_X$.

Next, we suppose that $f:X\to X$ is modelled by a Young tower $\Delta=Y^\tau$
with exponential tails as in Subsection~\ref{sec:YT}.
Define the {\em induced return time function}
\[
\varphi^Y:Y\to\Z^+, \qquad
\SMALL \varphi^Y=\sum_{\ell=0}^{\tau-1}\varphi\circ f^\ell.
\]
Assume that
$\varphi:X\to\Z^+$ is constant on $f^\ell\bar\pi^{-1}a$ for all
$0\le\ell<\tau(a)$ and all $a$.
Then
$\varphi^Y$ is well-defined on $\bar Y$ and constant on partition elements.

The final condition is somewhat technical and is based on~\cite[Lemma~5.4]{BalintGouezel06} which is itself based on~\cite[Sublemma, p.~612]{Young98}.
Given $h\in C^\eta(\Lambda)$, define
$H^Y=\sum_{\ell=0}^{\varphi^Y-1}h\circ T^\ell:Y\to\R$.
Let $\cB$ be the $\sigma$-algebra generated by $\cW^s$.
Then 
$\E(H^Y|\cB)=\zeta\circ \bar\pi$ where $\zeta\in L^1(\bar Y)$.
It is immediate that
\begin{align} \label{eq:zeta0}
|\zeta(y)|\le |h|_\infty\varphi^Y\!(a)
\quad\text{for all $y\in a$ and all $a$.}
\end{align}
We require that there are constants $C>0$, $\gamma_0\in(0,1)$ such that
\begin{align} \label{eq:zeta}
|\zeta(y)-\zeta(y')|\le C\varphi^Y\!(a)\gamma_0^{s(y,y')}
\quad\text{for all $y,y'\in a$ and all $a$.}
\end{align}
Under these assumptions, we say that $T:\Lambda\to\Lambda$ possesses a {\em Chernov-Markarian-Zhang structure}.

\begin{rmk}  The exponential tail condition for $\tau$ is assumed for convenience, but the abstract results require only that $\mu_Y(\tau>n)=O(n^{-q})$ for $q$ sufficiently large.
\end{rmk}

\begin{rmk}
The method of choosing a first return map modelled by a Young tower with exponential tails arises in various contexts in the literature, see for example~\cite{BruinLuzzattoStrien03,BruinTerhesiu18} in the noninvertible context.
However, the method plays a special role in the context of billiards as we now briefly recall.

Young~\cite{Young98} introduced Young towers with exponential tails as a general method for dealing with diffeomorphisms with singularities; the initial landmark application was to prove exponential decay of correlations for planar finite horizon dispersing billiards.  Chernov~\cite{Chernov99}  simplified the construction of exponential Young towers and used this to prove exponential decay of correlations for planar dispersing billiards with infinite horizon.
Then Young~\cite{Young99} studied examples with subexponential decay of correlations using Young towers with subexponential tails.  Markarian~\cite{Markarian04}, noting that Chernov's simplification no longer applies in the subexponential case, devised the method outlined in this section: namely to construct a first return map for which Chernov~\cite{Chernov99} applies.  This was used to prove the decay of correlations bound $O(1/n)$ for Bunimovich stadia.  The method was extended and simplified by
Chernov \& Zhang~\cite{ChernovZhang05}
who applied it to a large class of billiard examples.  Subsequent applications of the method include~\cite{ChernovMarkarian07,ChernovZhang05b} as well as Zhang~\cite{Zhang17a} who analysed the examples discussed in this paper.
\end{rmk}

\section{Statement of main results}
\label{sec:results}

Throughout this section, we suppose that $T:\Lambda\to\Lambda$ possesses a Chernov-Markarian-Zhang structure as in Section~\ref{sec:CMZ}, with first return map $f=T^\varphi:X\to X$ modelled by a Young tower with exponential tails.

For random elements $X_n$, $X$ taking values in a metric space, we write $X_n\to_w X$ if $\PP(X_n\in B)\to\PP(X\in B)$ for all Borel sets $B$ with
$\PP(X\in B)=0$.  When the metric space is $\R$, we write $\to_d$ instead of $\to_w$.  When the metric space is $\R$ and $X=0$, then this is equivalent to the simpler concept, convergence in probability, denoted $X_n\to_p0$.
For background on stable laws and L\'evy processes, we refer to~\cite{SamorodnitskyTaqqu}.

We assume that there exists $\alpha\in(1,2)$ such that
the first return time $\varphi:X\to\Z^+$ satisfies the limit law
\begin{align} \label{eq:stable}
\frac{1}{n^{1/\alpha}}\Big(\sum_{j=0}^{n-1}\varphi\circ f^j-n\int_X\varphi\,d\mu_X\Big)\to_d G \quad\text{on $(X,\mu_X)$},
\end{align}
where $G$ is an $\alpha$-stable law.
Since $\varphi\ge1$, this stable law
is totally skewed to the right.

Let $v:\Lambda\to\R$ be a H\"older observable with $\int_\Lambda v\,d\mu=0$.
Define the associated induced observable $V:X\to\R$ given by
$V(x)=\sum_{\ell=0}^{\varphi(x)-1}v(T^\ell x)$.
We assume that 
\begin{align} \label{eq:Lp}
V-I\varphi\in L^p(X)\quad\text{for some $I\neq 0$,\, $p>\alpha$}.
\end{align}
Define $v_n=\sum_{j=0}^{n-1}v\circ T^j:\Lambda\to\R$.

\begin{thm}[Stable law] \label{thm:stable}
Suppose that $T:\Lambda\to \Lambda$ possesses a Chernov-Markarian-Zhang structure and that
$v:\Lambda\to\R$ is a H\"older mean zero observable.
Assume~\eqref{eq:stable}  and~\eqref{eq:Lp}.  Then
$n^{-1/\alpha}v_n \to_d 
(\bar\varphi)^{-1/\alpha} I G$.
\end{thm}

Next, let $W\in D[0,\infty)$ be the $\alpha$-stable L\'evy process with $W(1)=_d G$.
Define $W_n:\Lambda\to D[0,\infty)$ by $W_n(t)=n^{-1/\alpha}v_{[nt]}$.

Define $M_1,\,M_2:X\to[0,\infty)$,
\begin{align*}
M_1 & = \max_{1\le\ell'\le\ell\le \varphi}  (v_{\ell'}-v_\ell)
	\wedge \max_{1\le\ell'\le\ell\le \varphi}  (v_\ell-v_{\ell'}),
\\
M_2 & = \Big\{\max_{0 \le\ell\le \varphi}   (-v_\ell)
	+ \max_{0 \le\ell\le \varphi}   (v_\ell-V) \Big\} 
	 \wedge 
	\Big\{\max_{0 \le\ell\le \varphi}    v_\ell 
	+ \max_{0 \le\ell\le \varphi}  (V-v_\ell) \Big\}.
\end{align*}
Note that  $M_1=0$ if and only if excursions between returns to $X$ are monotone~\cite{MZ15}, and $M_2=0$ if and only if excursions starting at 
$x\in X$ remain between $0$ and $V(x)$.

\begin{thm}[WIP] \label{thm:levy}
Suppose that $T:\Lambda\to \Lambda$ possesses a Chernov-Markarian-Zhang structure and that
$v:\Lambda\to\R$ is a H\"older mean zero observable.
Assume~\eqref{eq:stable}  and~\eqref{eq:Lp}.  
\begin{itemize}
\item[(a)] If
$n^{-1/\alpha}\max_{j\le n}M_1\circ f^j\to_p0$
on $(X,\mu_X)$,
then $W_n\to_w (\bar\varphi)^{-1/\alpha}I W$ 
on $(\Lambda,\mu)$ in $(D[0,\infty),\cM_1)$.
\item[(b)] If
$n^{-1/\alpha}\max_{j\le n}M_2\circ f^j\to_p0$
on $(X,\mu_X)$,
then $W_n\to_w (\bar\varphi)^{-1/\alpha} I W$
on $(\Lambda,\mu)$ in $(D[0,\infty),\cM_2)$.
\end{itemize}
\end{thm}

The theorem asserts that whenever excursions satisfy a mild monotonicity condition ($n^{-1/\alpha}\max_{j\le n}M_1\circ f^j\to_p0$),
or lie within a controlled distance from its endpoints ($n^{-1/\alpha}\max_{j\le n}M_2\circ f^j\to_p0$), then we obtain the WIP
in the $\cM_1$ or $\cM_2$ topology respectively. 

\begin{rmk} \label{rmk:SDC}
Let $(\Omega,\PP)$ be a probability space and $R_n:\Omega\to S$ a
sequence of Borel measurable maps where $S$ is a metric space.
{\em Strong distributional convergence} of $R_n$ to a random element $R$
on $(\Omega,\PP)$ means that $R_n\to_w R$ in $S$ on the probability 
space $(\Omega,\PP')$ for all probability measures $\PP'\ll \PP$.

In the context of Theorem~\ref{thm:levy}, strong distributional convergence 
on $(\Lambda,\mu)$ is automatic.  
Let $T$ be an ergodic measure-preserving transformation on a probability space $(\Lambda,\mu)$ and let $\mu'$ be an absolutely continuous probability measure.  Based on ideas of~\cite{Eagleson76},
it was shown in~\cite[Theorem~1 and Corollary~3]{Zweimuller07} that 
distributional convergence in $(D[0,\infty),\cJ_1)$ holds on $(\Lambda,\mu)$ if and only if it holds on $(\Lambda,\mu')$.
Hence
distributional convergence in $D[0,\infty)$ with the $\cJ_1$ topology on $(\Lambda,\mu)$ is equivalent to strong distributional convergence.  As pointed out in~\cite[Proposition~2.8]{MZ15}, this carries over immediately to weaker topologies on $D[0,\infty)$
such as $\cM_1$ and $\cM_2$.
\end{rmk}

\begin{rmk}  A more concise formula for $M_2$ can be obtained by noting that
\begin{align*}
M_2 & = 
  \Big\{-\min_{0 \le\ell\le \varphi}   v_\ell
	+ \max_{0 \le\ell\le \varphi}   v_\ell-V \Big\} 
	 \wedge 
	\Big\{\max_{0 \le\ell\le \varphi}   v_\ell 
	+ V- \min_{0 \le\ell\le \varphi}  v_\ell \Big\}
\\ & = \max_{0 \le\ell\le \varphi}   v_\ell - \min_{0 \le\ell\le \varphi}   v_\ell - |V|.
\end{align*}
\end{rmk}

The next result, proved in Section~\ref{sec:An}, extends and significantly improves~\cite[Proposition~2.7]{MZ15}.

\begin{prop} \label{prop:dom}
Let $i\in\{1,2\}$.  Suppose that there are constants 
$C>0$, $\delta\in(0,1)$ such that
 $M_i\le C\varphi^\delta$ almost everywhere.
 Then the assumption on $M_i$ in Theorem~\ref{thm:levy} is satisfied.
\end{prop}

In Section~\ref{sec:JZ}, we require the following converse result for the $\cM_1$ topology. (There  is no such converse result for $\cM_2$.)   

\begin{prop} \label{prop:converse}
If $W_n\to_w IW$ in $(D[0,\infty),\cM_1)$ for some constant $I\neq0$,
then $n^{-1/\alpha}\max_{0\le j\le n}M_1\circ f^j\to_p0$.
\end{prop}

\begin{proof}  
Without loss, $I=1$.
Fix $c>0$.  
Define $\Delta W(t)=W(t)-W(t_-)$.  
The stable law $G$ is totally skewed with L\'evy measure supported in $(0,\infty)$, so  
$\PP\{\Delta W(t)<-c\;\text{for some $0\le t\le 2\bar\varphi$}\}=0$.

For $\delta>0$, define
\[
E_\delta=\{u\in D[0,2\bar\varphi]:u(t)-u(t')<-c
\quad\text{for some $0\le t'<t<(t'+\delta)\wedge 2\bar\varphi$}\}.
\]
Since $W_n\to_w W$ in $\cM_1$, for any $\eps>0$ there exists $\delta>0$, $n_0\ge1$ such that
\[
\mu(W_n\in E_\delta)<\eps
\quad\text{for $n\ge n_0$.}
\]

Let $\varphi_n=\sum_{j=0}^{n-1}\varphi\circ f^j$.
Since $\varphi$ is integrable, it follows from the ergodic theorem that $n^{-1}\varphi_n\to\bar\varphi$ a.e.\ and so
$n^{-1}\varphi\circ f^n\to0$ a.e.  
It follows easily that $n^{-1}\max_{j\le n}\varphi\circ f^j\to0$ a.e.  
Hence there exists $n_1\ge n_0$ such that
\[
\mu\big(n^{-1}\max_{j\le n}\varphi\circ f^j\ge \delta\big)+
 \mu(n^{-1}\varphi_n\ge 2\bar\varphi)<\eps
\quad\text{for $n\ge n_1$.}
\]

Now,
\begin{align*}
n^{-1/\alpha}\max_{j\le n}M_1\circ f^j
& \le 
n^{-1/\alpha}\max_{j\le n}\max_{0\le\ell'<\ell<\varphi\circ f^j}(v_{\ell'}-v_{\ell})\circ f^j
 \le {\max}^* (W_n(t')-W_n(t)),
\end{align*}
where $\max^*$ is the maximum over
$0\le t'<t<(t'+n^{-1}\max_{j\le n}\varphi\circ f^j)\wedge n^{-1}\varphi_n$.

It follows that for $n\ge n_1$,
\begin{align*}
 & \mu\big\{n^{-1/\alpha}\max_{j\le n}M_1\circ f^j>c\big\}   \le 
 \mu\big\{{\max}^*(W_n(t')-W_n(t))>c\big\}  
\\ & \le 
\mu(n^{-1}\varphi_n\ge 2\bar\varphi)+\mu\big(n^{-1}\max_{j\le n}\varphi\circ f^j\ge \delta\big) +
 \mu\Big(\max_{0\le t'<t<(t'+\delta)\wedge 2\bar\varphi}(W_n(t')-W_n(t))> c\Big)
\\ & <\eps+\mu(W_n\in E_\delta)<2\eps.
\end{align*}
Hence $n^{-1/\alpha}\max_{j\le n}M_1\circ f^j\to_p0$.
\end{proof}

\begin{rmk} \label{rmk:PM}
Convergence results in the $\cM_1$ topology for nonuniformly hyperbolic maps were considered previously by~\cite{MZ15} with applications to  Markov Pomeau-Manneville intermittent maps~\cite{PomeauManneville80}.
Such maps fall into a greatly simplified version of the Chernov-Markov-Zhang framework.
Fix $\alpha\in(1,2)$ and set $\Lambda=[0,1]$.
A prototypical example~\cite{LSV99} is the map $T:\Lambda\to\Lambda$ given by
$Tx=\begin{cases} x(1+2^{1/\alpha}x^{1/\alpha}) & x<\frac12 \\ 2x-1 & x>\frac12
\end{cases}$, but the method applies equally to the general class of intermittent Markov maps considered by~\cite{Thaler95}.
Taking $X=[\frac12,1]$, the first return map $f=T^\varphi:X\to X$ is already Gibbs-Markov, so there is no need to consider an induced return map $F=T^{\varphi^Y}$, nor to quotient along stable leaves.  In other words, $X=Y=\bar Y$.
For these examples, condition~\eqref{eq:stable} holds by~\cite{Gouezel04}.
Condition~\eqref{eq:Lp} and condition~(a) in Theorem~\ref{thm:levy} were verified in~\cite[Section~4]{MZ15}.

Theorem~\ref{thm:levy}(a) also applies to non-Markovian intermittent maps $T:\Lambda\to\Lambda$: the so-called AFN maps studied by~\cite{Zweimuller98}.  
A specific example is 
given by $Tx=x(1+bx^{1/\alpha}) \bmod1$ which is not Markov   
when the positive constant $b$ is not an integer.  
As far as we know, the WIP for stable laws has not been previously studied for such maps.  
Since this is a much simpler situation than for our main billiard example, we just sketch the details.
(In fact, the situation lies in between those for Markov intermittent maps and billiards: quotienting along stable leaves is not required, but we do need to consider an induced map $F=T^{\varphi^Y}$.)

Take $X$ to be the interval of domain of the rightmost branch of $T$.  Let $v:\Lambda\to\R$ be H\"older with $v(0)\neq0$ and define $V=\sum_{\ell=0}^{\varphi-1}v\circ T^\ell$ where $\varphi:X\to\Z^+$ is the first return time.  The same calculations as in the Markov case show that
$\mu_X(\varphi>n)\sim cn^{-\alpha}$ for some $c>0$ and that $V-v(0)\varphi\in L^p(X)$ for some $p>\alpha$.  Hence~\eqref{eq:Lp} is satisfied.
Also condition~(a) of Theorem~\ref{thm:levy} holds as in the Markov case.
By~\cite[Section~9]{BruinTerhesiu18}, $f=T^\varphi$ is modelled by a Young tower with exponential tails so these maps fall into the Chernov-Markarian-Zhang framework.

It remains to verify 
the stable law~\eqref{eq:stable}.  One method is to proceed as in~\cite[Section~3]{JungZhangsub}, but alternatively we can make use of the fact proved in~\cite{BruinTerhesiu18}  that $\varphi^Y$ inherits the tail asymptotic satisfied by $\varphi$.
Since $F:Y\to Y$ is Gibbs-Markov and $\varphi^Y$ is constant on partition elements, a stable law for $\varphi^Y$ is immediate by~\cite[Theorem~6.1]{AaronsonDenker01}.  This yields the desired stable law for $\varphi$ by Theorem~\ref{thm:induce}.  
\end{rmk}

\section{Inducing functional limit laws}
\label{sec:Minduce}

The proof of Theorem~\ref{thm:levy}(a) makes use of a purely probabilistic
result~\cite[Theorem~2.2]{MZ15} on inducing functional limit laws on $D[0,\infty)$ with the $\cM_1$ topology.  The result in~\cite{MZ15} is stated in a slightly generalised form in Theorem~\ref{thm:M1induce} below.
The proof of Theorem~\ref{thm:levy}(b) makes use of the corresponding result
in the $\cM_2$ topology.
In this section, it is not required that $W$ is a L\'evy process.

We assume the set up in Section~\ref{sec:tower} but with different notation
(this simplifies the application of Theorem~\ref{thm:M1induce} in Sections~\ref{sec:An}
and~\ref{sec:proof}).
Let $S:\Omega\to\Omega$ be an ergodic measure-preserving transformation on a probability space $(\Omega,\mu_\Omega)$ 
and fix a positive measure subset $\Omega_0\subset\Omega$.  
Let $\mu_{\Omega_0}$ be a probability measure on $\Omega_0$ and let $r:\Omega_0\to\Z^+$ be an integrable return time 
such that the return map (not necessarily the first return) $S_0=S^r:\Omega_0\to\Omega_0$ is measure-preserving and ergodic.
Define $\bar r=\int_{\Omega_0}r\,d\mu_{\Omega_0}$.  
Let $\hS:\Delta\to\Delta$ denote the tower with base map $S_0$ and return time $r$, and let
$\pi:\Delta\to\Omega$ be the semiconjugacy $\pi(y,\ell)=S^\ell y$.  
We assume that 
$\mu_\Delta=(\mu_{\Omega_0}\times{\rm counting})/\bar r$ is ergodic and that
$\pi_*\mu_\Delta =\mu_\Omega$.

Let $\phi:\Omega\to\R$ be measurable, with induced observable
$\Phi : \Omega_0 \to \R$ given by $\Phi=\sum_{\ell=0}^{r-1} \phi\circ S^\ell$.
Let
$b_n$ be a sequence of positive numbers.
Define c\`adl\`ag processes $\psi_n$ on $\Omega$ and $\Psi_n$ on $\Omega_0$:
\[
\psi_n(t)=\frac1{b_n} \sum_{j=0}^{[nt]-1} \phi\circ S^j , \qquad
	\Psi_n(t)=\frac1{b_n} \sum_{j=0}^{[nt]-1} \Phi\circ S_0^j.
\]
Let $W\in D[0,\infty)$ and define $\tW(t)=W(\bar r t)$.
(If $W$ is an $\alpha$-stable L\'evy process, $\alpha\in(0,2]$, then $\tW=\bar r^{1/\alpha} W$.)
Also, define
$\phi_\ell=\sum_{j=0}^{\ell-1}\phi\circ S^j$ and
\begin{align*}
M_1 & = \max_{1\le\ell'\le\ell\le r}  (\phi_{\ell'}-\phi_\ell)
	\wedge \max_{1\le\ell'\le\ell\le r}  (\phi_\ell-\phi_{\ell'}),
\\
M_2 & = \Big\{\max_{0 \le\ell\le r}   (-\phi_\ell)
	+ \max_{0 \le\ell\le r}   (\phi_\ell-\Phi) \Big\} 
	 \wedge 
	\Big\{\max_{0 \le\ell\le r}    \phi_\ell 
	+ \max_{0 \le\ell\le r}  (\Phi-\phi_\ell) \Big\}.
\end{align*}

\begin{thm} \label{thm:M1induce} Let $i\in\{1,2\}$.
Suppose that on $(\Omega_0,\mu_{\Omega_0})$
\begin{enumerate}
\item 
	$\Psi_n \to_w \tW$ in $(D[0,\infty),\cM_i)$ and
\item $\frac1{b_n} \max_{0\le j \le n} M_i\circ S_0^j \to_p 0$.
\end{enumerate}
Then $\psi_n \to_w W$ in $(D[0,\infty),\cM_i)$ on $(\Omega,\mu_\Omega)$.
\end{thm}

\begin{pfof}{Theorem~\ref{thm:M1induce} for $i=1$.}
Under the additional assumptions that $b_n$ is regularly varying and $ r$ is the first return time, this is is precisely~\cite[Theorem~2.2]{MZ15}.
(The conclusion in~\cite[Theorem~2.2]{MZ15} is stated slightly differently using that $\bar r^{-1}=\mu_\Omega(\Omega_0)$ for first return times.)
It is easily checked that the proof in~\cite{MZ15} does not use any properties of the sequence $b_n$.   

It remains to drop the assumption that $r$ is the first return time to $\Omega_0$.
Note that $\Omega_0\subset \Omega$ is naturally identified with $\Delta_0=\{(y,0): y\in \Omega_0\}\subset\Delta$ and $ r:\Delta_0\to\R$ is now the first return to $\Delta_0$ for the dynamics on $\Delta$.  Define $\hS_0:\Delta_0\to\Delta_0$, $\hS_0(y,0)=(S_0y,0)$.

The observable $\phi:\Lambda\to\R$ lifts to an observable $\hat \phi=\phi\circ\pi:\Delta\to\R$.
Define the corresponding c\`adl\`ag process 
$\hpsi_n(t)=b_n^{-1}\sum_{j=0}^{[nt]-1}\hat \phi\circ \hS^j$ on $\Delta$.
Also, we define $\hPhi$, $\hPsi_n$, $\hM_1$ on $\Delta_0$ (corresponding to $\Phi$, $\Psi_n$, $M_1$ on $\Omega_0$) using $\hphi$, $\hS$, $\hS_0$ instead of $\phi$, $S$, $S_0$, so
\[
\hPhi=\sum_{\ell=0}^{\hat r-1}\hphi\circ \hS^\ell,
\quad
\hPsi_n(t)=\frac{1}{b_n}\sum_{j=0}^{[nt]-1}\hPhi\circ \hS_0^j, \quad
\hM_1=
\max_{1\le\ell'\le\ell\le \hat r}  (\hat \phi_{\ell'}-\hat \phi_\ell)
	\wedge \max_{1\le\ell'\le\ell\le \hat r}  (\hat \phi_\ell-\hat \phi_{\ell'}),
\]
where $\hat \phi_\ell=\sum_{j=0}^{\ell-1}\hat \phi\circ\hS^j$
and $\hat r(y,0)=r(y)$.

Note that
\[
\hPhi(y,0)=\Phi(y), \quad \hPsi_n(t)(y,0)=\Psi_n(t)(y), \quad \hM_1(y,0)=M_1(y).
\]
In particular, the assumptions 1 and 2 for 
$\Psi_n$ and $M_1$ on $\Omega_0$  imply the corresponding assumptions for
$\hPsi_n$ and $\hM_1$ on $\Delta_0$.
Since $\hat r:\Delta_0\to\Z^+$ is the first return time, 
$\hpsi_n\to_w W$ on $(\Delta,\mu_\Delta)$ in $(D[0,\infty),\cM_1)$
by~\cite[Theorem~2.2]{MZ15}.
The result follows 
since $\pi$ is a measure-preserving semiconjugacy.
\end{pfof}

\begin{pfof}{Theorem~\ref{thm:M1induce} for $i=2$.}
The strategy here is similar to the one of \cite[Theorem~2.2]{MZ15}.
As in the proof for $i=1$, by considering the associated tower we may suppose without loss that
$ r:\Omega_0\to\Z^+$ is the first return time.

Write $\psi_n = U_n + R_n$, where
\[
U_n(t)  = \frac1{b_n}\sum_{\ell=0}^{N_{[nt]}-1} \Phi \circ S_0^\ell
	\quad\text{and}\quad 
	R_n(t) = \frac1{b_n} \Bigg( \sum_{\ell=0}^{[nt]-  r_{N_{[nt]}}-1} \phi \circ S^\ell \Bigg) \circ S_0^{N_{[nt]}}.
\]
Here, $r_k=\sum_{j=0}^{k-1} r\circ S^j$ and $N_k(x)=\max \{ \ell \ge 1 : r_{\ell}(x) \le k \}$ is the number of returns of $x$ 
to the set $\Omega_0$, under iteration by $S$,
up to time $k$. 

By~\cite[Lemma~3.4]{MZ15}, $U_n\to_w (\tW(\bar r^{-1}t))_{t\ge0}=W$
in $(D[0,\infty),\cM_2)$.  (The hypotheses of~\cite[Lemma~3.4]{MZ15}
are with respect to the $\cM_1$ topology.  
However, most of the proof holds in any separable metric space and the only 
ingredient that relies on the specific topology is~\cite[Theorem~13.2.3]{Whitt} which is formulated for both $\cM_1$ and $\cM_2$.)

We claim that
$d_{\cM_2,[0,K]} (\psi_n, U_n)\to 0$ as $n\to\infty$
for each $K\in\N$.
Then by~\cite[Theorem~3.1]{Billingsley99}, 
$\psi_n\to_w W$ in $(D[0,K],\cM_2)$ 
for each $K\in\N$, and the result follows.

It remains to verify the claim.
This means taking into account the contribution of the final incomplete excursion from $\Omega_0$ (if any) encoded by $R_n$.
Following \cite[Lemmas~3.5 and~3.6]{MZ15},
given $x\in \Omega_0$, $n\ge1$, write $g_j(t)= \psi_n (t) (x) |_{[t_j,t_{j+1}]}$ for every $0\le j \le Kn+1$, where $t_j= \frac1n  r_j \wedge K$. Then 
\[
d_{\cM_2, [0,K]}(\psi_n(\cdot)(x), U_n(\cdot)(x)) 
	 \le \max_{0\le j \le Kn+1} \; d_{\cM_2, [t_j,t_{j+1}]}(g_j , \bar g_j),
\]
where $\bar g_j=U_n|_{[t_j,t_{j+1}]}=g_j|_{[t_j,t_{j+1})}+1_{\{t_{j+1}\}}g_j(t_{j+1})$.

By Lemma~\ref{lem:approx},
\begin{align*}
	d_{\cM_2, [t_j,t_{j+1}]}(g_j , \bar g_j) 
	& \le  t_{j+1}-t_j +  A_j\wedge B_j
	\le   \frac1n  r(S_0^jx) +  A_j\wedge B_j,
\end{align*}
where
\begin{align*}
A_j & = 
\sup_{t\in [t_j,t_{j+1}]} ( g_j(t_j) -g_j(t))+\sup_{t\in [t_j,t_{j+1}]}(g_j(t)-g_j(t_{j+1}))
 \\& = \sup_{t\in [t_j,t_{j+1}]}
 (\psi_n(t_j)(x)-\psi_n(t)(x))
+\sup_{t\in [t_j,t_{j+1}]}(\psi_n(t)(x)-\psi_n(t_{j+1})(x)) \\
 & = \frac{1}{b_n}\max_{0\le\ell\le r(S_0^jx)}(-\phi_{\ell}(S_0^jx))
 +\frac{1}{b_n}\max_{0\le\ell\le r(S_0^jx)}(\phi_\ell(S_0^jx)-\Phi(S_0^jx)),
\end{align*}
and similarly
\begin{align*}
B_j & = 
\sup_{t\in [t_j,t_{j+1}]} ( g_j(t) -g_j(t_j))+\sup_{t\in [t_j,t_{j+1}]}(g_j(t_{j+1})-g_j(t))
\\
& = \frac{1}{b_n}\max_{0\le\ell\le r(S_0^jx)}\phi_{\ell}(S_0^jx)
 +\frac{1}{b_n}\max_{0\le\ell\le r(S_0^jx)}(\Phi(S_0^jx)-\phi_{\ell}(S_0^jx)).
\end{align*}
In particular, $A_j\wedge B_j\le \frac{1}{b_n}M_2(S_0^jx)$.  Hence
we have shown that
\[
d_{\cM_2, [0,K]}(\psi_n, U_n)
	 \le \frac1n\max_{0\le j \le Kn+1} 
 r\circ S_0^j+\frac{1}{b_n}\max_{0\le j \le Kn+1}M_2\circ S_0^j.
\]
The first term converges to zero a.e.\ by ergodicity, and the second term
converges to zero in probability by the assumption on $M_2$.
\end{pfof}

\section{Limit laws for $\varphi$ and $\varphi^Y$}
\label{sec:An}

Recall that $T:\Lambda\to \Lambda$ is assumed to possess a Chernov-Markarian-Zhang structure,  with first return map $f=T^\varphi:X\to X$ modelled by a Young tower $\Delta=Y^\tau$ with exponential tails and
induced return time
$\varphi^Y=\sum_{\ell=0}^{\tau-1}\varphi\circ f^\ell:Y\to\Z^+$.

In this section, we show how to pass from the stable law~\eqref{eq:stable}
for $\varphi$ to a stable law for $\varphi^Y$ and WIPs for $\varphi$ and $\varphi^Y$.  We also prove Proposition~\ref{prop:dom}.

Note that $\int_Y \varphi^Y\,d\mu_Y=\bar\varphi\bar\tau$.
Define the centered return times
\[
\tvarphi=\varphi-\bar \varphi, \qquad
\tvarphi^Y=\varphi^Y-\tau\bar \varphi, \qquad
\widetilde{\varphi^Y}=\varphi^Y-\bar\tau\bar\varphi.
\]
Define c\`adl\`ag processes $A_n$ and $A^Y_n$ on $X$ and $Y$,
\begin{align} \label{eq:An}
\SMALL A_n(t)=n^{-1/\alpha}\sum_{j=0}^{[nt]-1}\tvarphi\circ f^j, \qquad
A_n^Y(t)=n^{-1/\alpha}\sum_{j=0}^{[nt]-1}\tvarphi^Y\circ F^j.
\end{align}

\begin{lemma} \label{lem:An}
Assume that~\eqref{eq:stable} holds and
let $W$ be the $\alpha$-stable L\'evy process corresponding to the
totally skewed $\alpha$-stable law~$G$ in~\eqref{eq:stable}.
Then
\\[.75ex]
(a) $n^{-1/\alpha}\sum_{j=0}^{n-1}\tvarphi^Y\circ F^j\to (\bar\tau)^{1/\alpha}G$
on $(Y,\mu_Y)$.
\\[.75ex]
(b) $n^{-1/\alpha}\sum_{j=0}^{n-1}\widetilde{\varphi^Y}\circ F^j\to (\bar\tau)^{1/\alpha}G$
on $(Y,\mu_Y)$.
\\[.75ex]
(c) $A_n^Y\to_w (\bar\tau)^{1/\alpha}W$ on $(Y,\mu_Y)$ in $(D[0,\infty),\cJ_1)$.
\\[.75ex]
(d) $A_n\to_w W$ on $(X,\mu_X)$ in $(D[0,\infty),\cM_1)$.
\end{lemma}

\begin{proof}
(a) Since $\tau:\bar Y\to\Z^+$ has exponential tails, we certainly have that
$\tau\in L^2$.  Also $\tau$ is constant on partition elements  and $\bar F:\bar Y\to\bar Y$ is Gibbs-Markov, so
it is standard (see for example~\cite[Theorem~1.5]{Gouezel10b}) that $n^{-1/2}(\sum_{j=0}^{n-1}\tau\circ \bar F^j-n\bar\tau)$ converges in distribution (to a possibly degenerate normal distribution).
Since $\alpha\in(1,2)$,
\begin{align} \label{eq:tau}
n^{-{1/\alpha}}\Big(\sum_{j=0}^{n-1}\tau\circ F^j-n\bar\tau\Big)=_d
n^{-{1/\alpha}}\Big(\sum_{j=0}^{n-1}\tau\circ \bar F^j-n\bar\tau\Big)\to_p0.
\end{align}
By assumption~\eqref{eq:stable}, the centered return time function $\tvarphi$ satisfies a stable law on $X$.
Hence condition~(a) in Theorem~\ref{thm:induce}
is satisfied with $b_n=n^{1/\alpha}$ and it follows from Theorem~\ref{thm:induce} and Remark~\ref{rmk:A2} that $\tvarphi^Y$ satisfies the required stable law on $Y$.
\\[.75ex]
(b)
By~\eqref{eq:tau} and part~(a), 
\begin{align*}
 n^{-1/\alpha}\sum_{j=0}^{n-1}\widetilde{\varphi^Y}\circ F^j  & =
n^{-1/\alpha}\Big(\sum_{j=0}^{n-1}\varphi^Y\circ F^j-n\bar\tau\bar\varphi\Big)
\\ & =
n^{-1/\alpha}\sum_{j=0}^{n-1}\tvarphi^Y\circ F^j
+\bar\varphi n^{-1/\alpha}\Big( \sum_{j=0}^{n-1}\tau\circ F^j -n\bar\tau\Big)
\to_d
(\bar\tau)^{1/\alpha}G.
\end{align*}
\\[.75ex]
(c)
Recall that $\tvarphi^Y$ is constant on partition elements of the 
Gibbs-Markov map $\bar F:\bar Y\to\bar Y$.
By part~(a),
$n^{-1/\alpha}\sum_{j=0}^{n-1}\tvarphi^Y\circ \bar F^j\to_d 
(\bar\tau)^{1/\alpha}G$.
By Gou\"ezel~\cite[Theorem~1.5]{Gouezel10b}, $\tvarphi^Y$ lies in the domain
of attraction of the stable law
$(\bar\tau)^{1/\alpha}G$  and hence has tails that are
regularly varying with index $\alpha$.  We have verified the hypotheses 
of Tyran-Kami{\'n}ska~\cite[Corollary~4.1]{TyranKaminska10}\footnote{The hypothesis ``exponentially continued fraction mixing'' in~\cite[Corollary~4.1]{TyranKaminska10} is automatic for full-branch Gibbs-Markov maps (see the discussion immediately after~\cite[Example~4.1]{TyranKaminska10}).}, and so deduce that
$A_n^Y\to_w (\bar\tau)^{1/\alpha}W$ in
the $\cJ_1$ topology.
\\[.75ex]
(d) We apply Theorem~\ref{thm:M1induce} with $i=1$ to pass from $A_n^Y:Y\to\R$ to
$A_n:X\to\R$ via the inducing time $\tau:Y\to\Z^+$.
(The spaces $\Omega_0\subset \Omega$ in Theorem~\ref{thm:M1induce} correspond to the
spaces $Y\subset X$ here.  Similarly $\phi$, $\Phi$, $\psi_n$, $\Psi_n$, $r$ are called $\tvarphi$, $\tvarphi^Y$, $A_n$, $A_n^Y$, $\tau$,
and the maps $S:\Omega\to \Omega$, $S_0:\Omega_0\to \Omega_0$ are called $f:X\to X$, $F:Y\to Y$.)

Condition 1 of Theorem~\ref{thm:M1induce} is immediate from part~(c).
Define $M_1:Y\to\R$,
\[
M_1  = \max_{1\le\ell'\le\ell\le \tau}  (\tvarphi_{\ell'}-\tvarphi_\ell)
        \wedge \max_{1\le\ell'\le\ell\le \tau}  (\tvarphi_\ell-\tvarphi_{\ell'}),
\]
where $\tvarphi_\ell=\sum_{j=0}^{\ell-1}\tvarphi\circ f^j$.
We claim that $n^{-1/\alpha}\max_{0\le j\le n}M_1\circ F^j\to0$ a.e.
This implies condition~2 of Theorem~\ref{thm:M1induce} and the result follows.

By positivity of $\varphi$,
\[
M_1\le \max_{1\le\ell'\le\ell\le \tau}  (\tvarphi_{\ell'}-\tvarphi_\ell)
= \max_{1\le\ell'\le\ell\le \tau}  \big\{(\varphi_{\ell'}-\varphi_\ell)-(\ell'-\ell)\bar\varphi\big\}\le\tau\bar\varphi.
\]
Since $\tau$ has exponential tails, it is certainly the case that
$\tau\in L^\alpha(Y)$.  By the ergodic theorem,
$n^{-1}\sum_{j=0}^{n-1}\tau^\alpha\circ F^j\to \int_Y\tau^\alpha\,d\mu_Y$ a.e.\ and so
$\tau\circ F^n=o(n^{1/\alpha})$ a.e.   It follows easily that
$\max_{j\le n}\tau\circ F^j=o(n^{1/\alpha})$ a.e.   
Hence $n^{-1/\alpha}\max_{j\le n}M_1\circ F^j\to0$ a.e.\ as required.
\end{proof}

\begin{cor} \label{cor:An}
Under the assumptions of Lemma~\ref{lem:An}, if $\delta\in(0,1)$  then
$n^{-1/\alpha}\max_{0\le j\le n}|\tvarphi|^\delta\circ f^j\to_p0$ on $(X,\mu_X)$.
\end{cor}

\begin{proof}
The functional $\chi:(D[0,\infty),\cM_1)\to\R$, $\chi(g)=\sup_{[0,1]}|g|$
is continuous so, by the continuous mapping theorem applied to Lemma~\ref{lem:An}(d), we have
$\chi(A_n)\to_w\chi(W)$ on $(X,\mu_X)$.
Hence
$ n^{-1/\alpha}\max_{0\le j\le n}|\tvarphi|\circ f^j=\chi(A_n)$ converges in distribution and so 
$ n^{-1/(\delta\alpha)}\max_{0\le j\le n}|\tvarphi|\circ f^j\to_p0$.
The result follows.
\end{proof}

\begin{pfof}{Proposition~\ref{prop:dom}}
We have $\varphi=\bar\varphi+\tvarphi\ll 1+|\tvarphi|$,
so $M_i\ll\varphi^\delta\ll 1+|\tvarphi|^\delta$.  Hence
$n^{-1/\alpha}\max_{j\le n} M_i\circ f^j
\ll n^{-1/\alpha}(1+\max_{j\le n}|\tvarphi|^\delta\circ f^j)\to_p0$
by Corollary~\ref{cor:An}.
\end{pfof}

\begin{rmk} \label{rmk:intphi}
As seen in the proof of Lemma~\ref{lem:An}(c), $\tvarphi^Y$ lies in the domain of attraction of an $\alpha$-stable law,
so $\varphi^Y\in L^q(Y)$ for all $q<\alpha$.  It follows easily that
$\varphi\in L^q(X)$ for all $q<\alpha$.
\end{rmk}

\section{Moment estimates for induced observables}
\label{sec:moment}

In this section, we consider estimates for certain induced observables.
We continue to assume that $T:\Lambda\to\Lambda$ possesses a Chernov-Markarian-Zhang structure.   Our method follows~\cite[Section~5]{BalintGouezel06}.

\begin{prop} \label{prop:Lp}  
Let $H:X\to\R$ and suppose that $H\in L^q(X)$ for some $q>1$.
Define $H^Y=\sum_{\ell=0}^{\tau-1}H\circ f^\ell$.
Then $H^Y\in L^p(Y)$ for all $p<q$.
\end{prop}

\begin{proof}
Let $a>1$ with $1/a+1/q=1/p$.  Let $c'=c/a$, where $c>0$ is given by \eqref{eq:exptails}.  By H\"older's inequality,
\begin{align*}
{|H^Y|}_{L^p(Y)} & \ll
\sum_{n\ge1}\Big|1_{\{\tau=n\}}\sum_{\ell=0}^{n-1}H\circ f^\ell\Big|_{L^p(Y)}
 \le 
\sum_{n\ge1}\mu_Y(\tau=n)^{1/a}\Big|1_{\{\tau=n\}}\sum_{\ell=0}^{n-1}H\circ f^\ell\Big|_{L^q(Y)}
\\ & \le 
\sum_{n\ge1}e^{-c'n}\sum_{\ell=0}^{n-1}\Big|1_{\{\tau=n\}}H\circ f^\ell\Big|_{L^q(Y)}
 \ll 
\sum_{n\ge1}e^{-c'n}n {|H|}_{L^q(X)}
\ll {|H|}_{L^q(X)}<\infty,
\end{align*}
as required.
\end{proof}

\begin{lemma} \label{lem:mart}
Let $p\in(1,2]$.
Let $h\in C^\eta(\Lambda)$ with induced observables 
\[
H=\sum_{\ell=0}^{\varphi-1}h\circ T^\ell:X\to\R, \qquad
\tH^Y=\sum_{\ell=0}^{\tau-1}\tH\circ f^\ell:Y\to\R,
\]
where $\tH=H-\int_X H\,d\mu_X$.
Suppose that $H\in L^q(X)$ for some $q>p$.  Then
$\big|\max_{j\le n}|\sum_{i=0}^{j-1}\tH^Y\circ F^i|\big|_p\ll n^{1/p}$.
\end{lemma}

\begin{proof}
Following~\cite{BalintGouezel06}, we apply a Gordin type argument~\cite{Gordin69} to obtain an $L^p$ martingale-coboundary decomposition.

First, by Proposition~\ref{prop:Lp} we may suppose that $\tH^Y\in L^q(Y)$ for some (smaller) $q>p$.
Let $\overline{\cB}$ denote the underlying $\sigma$-algebra on $\bar Y$
and let $\cB=\bar\pi^{-1}\overline{\cB}$.
Then $\{F^n\cB,\, n\in\Z\}$ defines an increasing sequence of
$\sigma$-algebras on $Y$.
We claim that
\begin{align}
\label{eq:mart}
& \sum_{n=0}^\infty \big|\E(\tH^Y|F^n\cB)-\tH^Y\big|_p<\infty, \qquad
 \sum_{n=1}^\infty \big|\E(\tH^Y|F^{-n}\cB)\big|_p<\infty.
\end{align}

Suppose that the claim is true.  Then equivalently,
\[
 \sum_{n=0}^\infty \big|\E(\tH^Y\circ F^n|\cB)-\tH^Y\circ F^n\big|_p<\infty, \qquad
 \sum_{n=1}^\infty \big|\E(\tH^Y\circ F^{-n}|\cB)\big|_p<\infty,
\]
so the series 
\[
\chi=\sum_{n=0}^\infty (\E(\tH^Y\circ F^n|\cB)-\tH^Y\circ F^n)
+\sum_{n=1}^\infty \E(\tH^Y\circ F^{-n}|\cB),
\]
converges in $L^p(Y)$.
Define 
\begin{align} \label{eq:m}
m=\tH^Y+\chi-\chi\circ F\in L^p(Y).
\end{align}
Then
\begin{align} \label{eq:mg}
\SMALL m=\sum_{n=-\infty}^\infty (g_n-g_n\circ F)
=\sum_{n=-\infty}^\infty (g_{n+1}-g_n\circ F),
\end{align}
where $g_n=\E[\tH^Y\circ F^n|\cB]$.

Now, $g_n$ is $\cB$-measurable, while
$g_n\circ F$ is measurable with respect to $F^{-1}\cB\subset \cB$.
Hence $m$ is $\cB$-measurable.  Next, 
$g_n\circ F=\E[\tH^Y\circ F^n|\cB]\circ F=
\E[\tH^Y\circ F^{n+1}|F^{-1}\cB]$.
It follows that
\[
\E[g_n\circ F|F^{-1}\cB]=\E[\tH^Y\circ F^{n+1}|F^{-1}\cB]
=\E[\E[\tH^Y\circ F^{n+1}|\cB]|F^{-1}\cB]
=\E[g_{n+1}|F^{-1}\cB],
\]
where we used again that $F^{-1}\cB\subset \cB$.
Substituting into~\eqref{eq:mg},
we obtain \mbox{$\E[m|F^{-1}\cB]=0$}.  
Hence $\{m\circ F^{-n};n\in\Z\}$ is a martingale difference sequence with respect to the filtration $\{F^n\cB;n\in\Z\}$.

By Burkholder's inequality~\cite[Theorem~3.2]{Burkholder73},
\begin{align*}
\Big|\sum_{j=1}^{n}m\circ F^{-j}\Big|_p
& \ll \Big|\Big(\sum_{j=1}^{n}m^2\circ F^{-j}\Big)^{1/2}\Big|_p
=\Big(\int\Big(\sum_{j=1}^{n} m^2\circ F^{-j}\Big)^{p/2}\Big)^{1/p}
\\ & \le \Big(\int\sum_{j=1}^{n} |m|^p\circ F^{-j}\Big)^{1/p}
= |m|_p\,n^{1/p}.
\end{align*}
By Doob's inequality~\cite{Doob} (see also~\cite[Equation~(1.4), p.~20]{Burkholder73}),
\[
\Big|\max_{j\le n}\big|\sum_{i=0}^{j-1}m\circ F^i\big|\Big|_p
\le
2\Big|\max_{j\le n}\big|\sum_{i=1}^{j}m\circ F^{-i}\big|\Big|_p
\ll n^{1/p}.
\]
Also,
\[
\int_Y\max_{j\le n}|\chi\circ F^j-\chi|^p\le
2\sum_{j=0}^n|\chi\circ F^j|_p^p=2(n+1)|\chi|_p^p,
\]
so
$\big|\max_{j\le n}|\chi\circ F^j-\chi|\big|_p\ll n^{1/p}$.

By~\eqref{eq:m}, $\sum_{j=0}^{n-1}\tH^Y\circ F^j=
\sum_{j=0}^{n-1}m\circ F^j+\chi\circ F^n-\chi$, so the desired estimate for
$\tH^Y$ follows from the estimates for $m$ and $\chi$.

It remains to verify the claim.  The argument is identical to the one 
in~\cite[Lemma~5.3]{BalintGouezel06} except for the order of integrability.
 Hence we only sketch the argument referring 
to~\cite{BalintGouezel06}
  for the details (especially the prerequisite
 estimates for systems modelled by Young towers).

If $y,y'\in Y$ lie in the same stable leaf, then $|\tH^Y(y)-\tH^Y(y')|\ll \varphi^Y\!(y) d(y,y')^\eta$.
By~\eqref{eq:Ws},
the atoms of $F^n\cB$ have diameter at most $C\gamma_0^n$ for some $C>0$,
$\gamma_0\in(0,1)$.   Hence setting $\gamma=\gamma_0^\eta$, we have
(cf.~\cite[Estimate~(54)]{BalintGouezel06})
\[
|\tH^Y-\E(\tH^Y|F^n\cB)|\ll \varphi^Y \gamma^n.
\]
Choose $r>1$ with $1/r+1/q=1/p$.  In the case that the inducing time is large,
\[
\big|1_{\{\varphi^Y>n^{2r}\}}\tH^Y\big|_p
	\le  \mu(\varphi^Y>n^{2r})^{1/r}|\tH^Y|_q
	\le n^{-2}|\varphi^Y|_1^{1/r}|\tH^Y|_q,
\]
and similarly,
\[
\big|1_{\{\varphi^Y>n^{2r}\}}\E(\tH^Y|F^n\cB)\big|_p\le 
n^{-2}|\varphi^Y|_1^{1/r}|\E(\tH^Y|F^n\cB)|_q
\le n^{-2}|\varphi^Y|_1^{1/r}|\tH^Y|_q.
\]
Hence 
$\big|1_{\{\varphi^Y > n^{2r}\}}\{\tH^Y-\E(\tH^Y|F^n\cB)\}\big|_p
\le 2n^{-2}|\varphi^Y|_1^{1/r}|\tH^Y|_q$.
On the other hand,
\[
\big|1_{\{\varphi^Y\le n^{2r}\}}\{\tH^Y-\E(\tH^Y|F^n\cB)\}\big|_\infty\ll n^{2r}\gamma^n.
\]
Combining the last two estimates, we obtain the first part of~\eqref{eq:mart}.

Next, write 
$\E(\tH^Y|\cB)=\zeta\circ\bar\pi$ where
$\zeta\in L^1(\bar Y)$.
Let 
$P:L^1(\bar Y)\to L^1(\bar Y)$ be the transfer operator associated to $\bar F$ 
(so $\int_{\bar Y}Pv\,w\,d\bar\mu_Y=\int_{\bar Y}v\,w\circ \bar F\,d\bar\mu_Y$ for all $w\in L^\infty(\bar Y)$).
By standard methods (see for example~\cite[Corollary~2.3(a)]{MN05}), it follows from integrability of $\varphi^Y$ and the estimates~\eqref{eq:zeta0} and~\eqref{eq:zeta} that there exist constants $C>0$, $\gamma\in(0,1)$ such that $|P^n\zeta|_\infty\le C\gamma^n$.
Moreover $\E(\cdot|\bar F^{-n}\overline{\cB})=(UP)^n=U^nP^n$
where $Uv=v\circ \bar F$.
Hence
\begin{align*}
\E(\tH^Y|F^{-n}\cB) & =
\E(\E(\tH^Y|\cB)|F^{-n}\cB) 
=\E(\zeta\circ\bar\pi|F^{-n}\cB) 
\\ & =\E(\zeta\circ\bar\pi|\bar\pi^{-1}\bar F^{-n}\overline{\cB}) 
=\E(\zeta|\bar F^{-n}\overline{\cB}) \circ\bar\pi
=(U^nP^n\zeta)\circ\bar\pi,
\end{align*}
and so 
\[
\big|\E(\tH^Y|F^{-n}\cB)\big|_{L^p(Y)}
=|U^nP^n\zeta|_{L^p(\bar Y)}
=|P^n\zeta|_{L^p(\bar Y)} \le |P^n\zeta|_\infty\ll\gamma^n.
\]
Hence $\big|\E(\tH^Y|F^{-n}\cB)\big|_{L^p(Y)}$ is summable, completing the proof of~\eqref{eq:mart}.~
\end{proof}

\section{Proof of Theorems~\ref{thm:stable} and~\ref{thm:levy}}
\label{sec:proof}

In this section, we complete the proof of the main results in Section~\ref{sec:results}.   

\begin{pfof}{Theorem~\ref{thm:stable}}
Define the H\"older observable $h=v-I:\Lambda\to\R$.
As in the statement of Lemma~\ref{lem:mart},
define $H=V-I\varphi$, $\tH=V-I\tvarphi$, 
$\tH^Y=V^Y-I\tvarphi^Y$. 
By~\eqref{eq:Lp}, $H\in L^p(X)$ for some $p>\alpha$ and hence
by Lemma~\ref{lem:mart},
\begin{align} \label{eq:E}
n^{-1/\alpha}\max_{j\le n}\Big|\sum_{i=0}^{j-1}\tH^Y\circ F^i\Big|
\to_p0.
\end{align}
Hence by Lemma~\ref{lem:An}(a),
\begin{align} \label{eq:stable2}
n^{-1/\alpha}\sum_{j=0}^{n-1}V^Y\circ F^j=
n^{-1/\alpha}\sum_{j=0}^{n-1}(I\tvarphi^Y+\tH^Y)\circ F^j\to_d (\bar\tau)^{1/\alpha} I G.
\end{align}

As in the proof of Theorem~\ref{thm:M1induce}, we can suppose without loss that 
$F=T^{\varphi^Y}:Y\to Y$ is a first return map.  
As a consequence of~\eqref{eq:stable2}  and Lemma~\ref{lem:An}(b) we can apply~\cite[Theorem~A.1]{Gouezel07} (see Remark~\ref{rmk:induce}) and it follows that
$n^{-1/\alpha}\sum_{j=0}^{n-1}v\circ T^j\to_d (\int_Y\varphi^Y\,d\mu_Y)^{-1/\alpha}(\bar\tau)^{1/\alpha} I G=(\bar\varphi)^{-1/\alpha} I G$.
(In applying Remark~\ref{rmk:induce}, it should be noted that $T:\Lambda\to\Lambda$, $v$, $\varphi^Y$ are called $f:X\to X$, $V$, $\tau$ in Appendix~\ref{app:induce}.)
\end{pfof}

Recall that
$W^Y_n(t)=n^{-1/\alpha}\sum_{j=0}^{[nt]-1}V^Y\circ F^j$ 
is a c\`adl\`ag process on $Y$.

\begin{lemma}[WIP on $Y$] \label{lem:J1}
Under the assumptions of Theorem~\ref{thm:stable},
$W^Y_n\to_w (\bar\tau)^{1/\alpha} I W$
on $(Y,\mu_Y)$ in $(D[0,\infty),\cJ_1)$.
\end{lemma}

\begin{proof}
Write $V^Y=I\tvarphi^Y+\tH^Y$ as in the proof of Theorem~\ref{thm:stable}.
Then $W_n^Y=IA_n^Y+B_n^Y$ where $A_n^Y$ is as in~\eqref{eq:An} and
$B_n^Y(t)=n^{-1/\alpha}\sum_{j=0}^{[nt]-1}\tH^Y\circ F^j$.
By~\eqref{eq:E}, for every $K>0$ one has that
$\sup_{[0,K]}|B_n^Y| \ll n^{-1/\alpha}\max_{j\le Kn}|\sum_{i=0}^{j-1}\tH^Y\circ F^i|\to_p0$.
Hence the result follows from Lemma~\ref{lem:An}(c).
\end{proof}

Next,  recall that
$W^X_n(t)=n^{-1/\alpha}\sum_{j=0}^{[nt]-1}V\circ f^j$ 
is a  c\`adl\`ag process 
on $X$.

\begin{lemma}[WIP on $X$] \label{lem:XM1}
Under the assumptions of Theorem~\ref{thm:stable},
$W^X_n\to_w I W$ on $(X,\mu_X)$ in $(D[0,\infty),\cM_1)$.
\end{lemma}

\begin{proof}
We apply Theorem~\ref{thm:M1induce} (with $i=1$) to pass from $V^Y:Y\to\R$ to
$V:X\to\R$ via the inducing time $\tau:Y\to\Z^+$.
(The spaces $\Omega_0\subset \Omega$ in Theorem~\ref{thm:M1induce} correspond to the
spaces $Y\subset X$ here.  Similarly $\phi$, $\Phi$, $\psi_n$, $\Psi_n$, $r$ are called $V$, $V^Y$, $W_n^X$, $W_n^Y$, $\tau$,
and the maps $S:\Omega\to \Omega$, $S_0:\Omega_0\to \Omega_0$ are called $f:X\to X$, $F:Y\to Y$.)

Condition 1 of Theorem~\ref{thm:M1induce} is immediate from Lemma~\ref{lem:J1}.
Define $M_1:Y\to\R$,
\[
M_1  = \max_{1\le\ell'\le\ell\le \tau}  (V_{\ell'}-V_\ell)
        \wedge \max_{1\le\ell'\le\ell\le \tau}  (V_\ell-V_{\ell'}),
\]
where $V_\ell=\sum_{j=0}^{\ell-1}V\circ f^j$.
We claim that $n^{-1/\alpha}\max_{0\le j\le n}M_1\circ F^j\to0$ in $L^1(Y)$.
This implies condition~2 of Theorem~\ref{thm:M1induce}
and the result follows.

It remains to verify the claim.  
Recall that $V=I\varphi+H$ and correspondingly $V^Y=I\varphi^Y+H^Y$.
Define $H_\ell=\sum_{j=0}^{\ell-1}H\circ f^j$
and $H^*=|H|^Y=\sum_{\ell=0}^{\tau-1}|H|\circ f^\ell$.
By assumption~\eqref{eq:Lp} and Proposition~\ref{prop:Lp},
$H^*\in L^p(Y)$ for some $p>\alpha$.

Suppose that $I>0$ (the case $I<0$ is similar).
Then $I\varphi>0$ and
\begin{align*}
M_1\le \max_{1\le\ell'\le\ell\le \tau}  (V_{\ell'}-V_\ell)\le 
\max_{1\le\ell'\le\ell\le \tau}  (H_{\ell'}-H_\ell)\le H^*.
\end{align*}
Hence 
\begin{align*}
\int_Y(\max_{j\le n}M_1\circ F^j)^p\,d\mu\le 
\int_Y \sum_{j=0}^n (H^*\circ F^j)^p\,d\mu=(n+1)\int_Y {H^*}^p\,d\mu\ll n,
\end{align*}
and so
$n^{-1/p}\max_{j\le n}M_1\circ F^j$ is bounded in $L^p(Y)$ proving the claim.
\end{proof}

\begin{pfof}{Theorem~\ref{thm:levy}}
We apply Theorem~\ref{thm:M1induce} to pass from $V:X\to\R$ to $v:\Lambda\to\R$ via the return time $\varphi:X\to\Z^+$.
(This time, the spaces $\Omega_0\subset \Omega$ in Theorem~\ref{thm:M1induce} correspond to the
spaces $X\subset \Lambda$ here.  Similarly $\phi$, $\Phi$, $\psi_n$, $\Psi_n$, $r$ are called $v$, $V$, $W_n$, $W_n^X$, $\varphi$,
and the maps $S:\Omega\to \Omega$, $S_0:\Omega_0\to \Omega_0$ are called $T:\Lambda\to \Lambda$, $f:X\to X$.
Also, $M_1$ and $M_2$ are defined as in Section~\ref{sec:results}.)
\\[.75ex]
(a)
Conditions~1 and~2 of Theorem~\ref{thm:M1induce} with $i=1$ correspond to
Lemma~\ref{lem:XM1} and
the assumption on $M_1$ respectively.
\\[.75ex]
(b)
Lemma~\ref{lem:XM1} asserts convergence in the $\cM_1$ topology and hence in the $\cM_2$ topology, so condition~1 of Theorem~\ref{thm:M1induce} ($i=2$) is satisfied.
Condition~2 of Theorem~\ref{thm:M1induce} corresponds to
the assumption on $M_2$.
\end{pfof}

\section{Billiards with cusps at flat points}
\label{sec:JZ}

We consider the Jung \& Zhang example~\cite{JungZhangsub} described in Example~\ref{ex:JZ}.
Zhang~\cite{Zhang17a} showed that such billiard maps $T:\Lambda \to \Lambda$ fit in the Chernov-Markarian-Zhang framework with first return map
$f=T^\varphi:X\to X$ where $X=(\Gamma_3\cup\dots\cup \Gamma_{n_0})\times[0,\pi]$.   
Recall that $\alpha=\frac{\beta}{\beta-1}\in(1,2)$.   
Define $I_v(s)$ as in~\eqref{eq:Iv} 
for continuous functions $v:\Lambda\to\R$.

In the remainder of this section, we fix $v:\Lambda\to\R$ H\"older continuous with mean zero such that $I_v(\pi)>0$.
Define the strictly increasing, hence invertible, function $\Psi(s)= I_1(\pi)^{-1}I_1(s)$, $s\in[0,\pi]$.

\begin{prop} \label{prop:I}
Let $\delta=\eta/(\beta-1)$ where $\eta$ is the H\"older exponent of $v$.
There is a constant $C>0$ such that
for all $0\le\ell\le \varphi(x)$, $x\in X$,
\[
v_\ell(x)=\varphi(x) I_1(\pi)^{-1} I_v\circ\Psi^{-1}(\ell/\varphi(x)) + E_\ell(x), \qquad
|E_\ell(x)|
\le C\varphi(x)^{1-\delta}.
\]
\end{prop}

\begin{proof}
Let $\tv(\theta)=\frac12\{v(r',\theta)+v(r'',\pi-\theta)\}$.
Proceeding as in~\cite[Section 6: Proof of Lemma~4.4]{JungZhangsub}, for $0\le\ell\le \varphi/2$,
\begin{align*}
v_\ell   
=\sum_{j=0}^{\ell-1}v\circ T^j
& =\sum_{j=0}^{\ell-1}\tv\circ\Psi^{-1}(j/\varphi)+O(\varphi^{1-\delta})
 =\varphi\int_0^{\ell/\varphi}\tv\circ\Psi^{-1}\,d\theta+O(\varphi^{1-\delta})
\\ & =\varphi I_1(\pi)^{-1}\int_0^{\Psi^{-1}(\ell/\varphi)}\tv(\theta)(\sin\theta)^{1/\alpha}\,d\theta+O(\varphi^{1-\delta})
 \\ & =\varphi I_1(\pi)^{-1} I_v\circ\Psi^{-1}(\ell/\varphi)+O(\varphi^{1-\delta}).
\end{align*}
In particular, $v_{\varphi/2}=
 \varphi I_1(\pi)^{-1} I_v(\pi/2)+O(\varphi^{1-\delta})$.

For $\varphi/2\le \ell\le \varphi$, using time reversibility and the estimates in~\cite{JungZhangsub}, 
\begin{align} \label{eq:I}  \nonumber
v_\varphi  -v_\ell  =\sum_{j=\ell}^{\varphi-1} v\circ T^j
  & =\varphi\int_{\ell/\varphi}^{1}\tv\circ \Psi^{-1}\,d\theta+O(\varphi^{1-\delta})
 \\ & =\varphi I_1(\pi)^{-1}\int_{\Psi^{-1}(\ell/\varphi)}^{\Psi^{-1}(1)}\tv(\theta)(\sin\theta)^{1/\alpha}\,d\theta+O(\varphi^{1-\delta})
  \\ &  =\varphi I_1(\pi)^{-1}\{I_v(\pi)-I_v\circ\Psi^{-1}(\ell/\varphi)\}+O(\varphi^{1-\delta}).
\nonumber
\end{align}
In particular, $v_\varphi-v_{\varphi/2}=\varphi I_1(\pi)^{-1}\{I_v(\pi)-I_v(\pi/2)\}+O(\varphi^{1-\delta})$, so
$v_\varphi=\varphi I_1(\pi)^{-1} I_v(\pi)+O(\varphi^{1-\delta})$.
Substituting the final estimate into~\eqref{eq:I} completes the proof.
\end{proof}
\begin{lemma} \label{lem:JZ}
Conditions~\eqref{eq:stable} and~\eqref{eq:Lp} hold.
\end{lemma}

\begin{proof}
Condition~\eqref{eq:stable} holds by~\cite[Theorem~3.1]{JungZhangsub}.
Taking $\ell=\varphi(x)$ in Proposition~\ref{prop:I},
$V=I\varphi+E$ where $I=I_v(\pi)/I_1(\pi)$ and
$|E|\ll \varphi^{1-\delta}$ for some $\delta>0$.  
Choose $p\in(\alpha,\alpha/(1-\delta))$.   Then 
$\varphi^{(1-\delta)p}$ is integrable by Remark~\ref{rmk:intphi} and
$\int_X |E|^p\,d\mu_X\ll
\int_X \varphi^{(1-\delta)p}\,d\mu_X$,
so~\eqref{eq:Lp} is satisfied.
\end{proof}

Note that Theorem~\ref{thm:JZ} is an immediate
consequence of Theorem~\ref{thm:stable} and Lemma~\ref{lem:JZ}.

\begin{cor} \label{cor:JZus}
The conditions on $I_v(s)$ in Theorems~\ref{thm:JZus1} and~\ref{thm:JZus2} are sufficient for the WIP.
In particular, Theorem~\ref{thm:JZus2} holds.
\end{cor}

\begin{proof}
We verify the assumptions of Theorem~\ref{thm:levy}.
Conditions~\eqref{eq:stable} and~\eqref{eq:Lp} hold by Lemma~\ref{lem:JZ}.
Hence it suffices to prove that 
$n^{-1/\alpha}\max_{j\le n}M_i\circ f^j\to_p0$ where $i\in\{1,2\}$ respectively.

First suppose that $s\mapsto I_v(s)$ is nondecreasing.
Note that $\Psi$ is increasing, so $v_\ell-E_\ell$ is a nondecreasing function of $\ell$.
Hence by Proposition~\ref{prop:I}, 
\[
M_1\le 
\max_{1\le\ell'\le\ell\le \varphi}  (v_{\ell'}-v_\ell)
\le \max_{1\le\ell'\le\ell\le \varphi}  (E_{\ell'}-E_\ell)
\le 2C\varphi^{1-\delta}.
\]
By Proposition~\ref{prop:dom},
$n^{-1/\alpha}\max_{j\le n}M_1\circ f^j\to_p0$.

Next suppose that $I_v(s)\in[0,I_v(\pi)]$ for all $s$.
Then $v_\ell\ge E_\ell$ and $v_\varphi-v_\ell\ge E_\varphi-E_\ell$
for $0\le \ell\le\varphi$.
By Proposition~\ref{prop:I},
\[
M_2  \le \max_{0 \le\ell\le \varphi}   (-v_\ell)
	+ \max_{0 \le\ell\le \varphi}   (v_\ell-v_\varphi) 
\le \max_{0 \le\ell\le \varphi}   (-E_\ell)
	+ \max_{0 \le\ell\le \varphi}   (E_\ell-E_\varphi) 
\le 3C\varphi^{1-\delta}.
\]
Again, it follows from Proposition~\ref{prop:dom} that
$n^{-1/\alpha}\max_{j\le n}M_2\circ f^j\to_p0$.
\end{proof}

It remains to prove necessity of the conditions for the WIP in Theorem~\ref{thm:JZus1}.
We require one further result from~\cite{JungZhangsub}.
\begin{prop} \label{prop:notprob}
$n^{-1/\alpha}\max_{0\le j\le n}\varphi\circ f^j\not\to_p0$.
\end{prop}

\begin{proof}
Define $\nu_n=\sum_{j=1}^n \delta_{n^{-1/\alpha}\varphi\circ f^j}$.
This is the expression in~\cite[Eq.~(3.15)]{JungZhangsub}.
By~\cite[Eq.~(3.7) and Section~3.2]{JungZhangsub}, $\mu(\nu_n((1,\infty))=0)\to c<1$.
In particular, 
$\mu(n^{-1/\alpha}\max_{j\le n}\varphi\circ f^j>1)=
\mu(\nu_n(1,\infty)\ge1)\not\to0$.
\end{proof}

\begin{pfof}{Theorem~\ref{thm:JZus1}}
Suppose that $I_v(s)$ is not monotone.  Then
there exists $0<s_1<s_2<\pi$ such that $I_v(s_2)<I_v(s_1)$.

For each $x\in X$,
set $\ell_r(x)=[\varphi(x)\Psi(s_r)]$ for $r=1,2$.
Then $0\le\ell_1\le \ell_2\le \varphi$.
By Proposition~\ref{prop:I},
$v_{\ell_r}=\varphi I_v(s_r)+O(\varphi^{1-\delta})$, so
\[
v_{\ell_1}-v_{\ell_2}=c_1\varphi+O(\varphi^{1-\delta}),\qquad
V=c_2\varphi+O(\varphi^{1-\delta}),
\]
where $c_1,c_2>0$.
Hence $M_1\ge c\varphi+O(\varphi^{1-\delta})$ where
$c=c_1\wedge c_2>0$.
By Proposition~\ref{prop:notprob},
$n^{-1/\alpha}\max_{j\le n}M_1\circ f^j\not\to_p0$.
By Proposition~\ref{prop:converse}, $W_n\not\to_w W$ in $\cM_1$.
The other direction was proved in 
Corollary~\ref{cor:JZus} so this completes the proof.
\end{pfof}

\appendix

\section{Inducing stable laws in both directions}
\label{app:induce}

We assume the set up from Section~\ref{sec:tower} with measure-preserving transformations $f$, $F=f^\tau$ and $\hf$ on probability spaces
$(X,\mu_X)$, $(Y,\mu_Y)$ and $(\Delta,\mu_\Delta)$ respectively.
Let $\pi:\Delta\to X$ be the measure-preserving semiconjugacy $\pi(y,\ell)=f^\ell y$ and set $\bar \tau=\int_Y\tau\,d\mu_Y$.  
We assume in addition that the probability measures $\mu_X$, $\mu_Y$, $\mu_\Delta$ are ergodic.

In the following result, based on~\cite{Gouezel07,MT04}, we relate limit theorems on $X$ and $Y$.

\begin{thm} \label{thm:induce}
Let $V\in L^1(X)$ with $\int_X V\,d\mu_X=0$.
Define the induced observable 
\[
\SMALL V^Y:Y\to\R,\qquad V^Y=\sum_{\ell=0}^{\tau-1}V\circ f^\ell,
\]
and the Birkhoff sums
\[
\SMALL V_n=\sum_{j=0}^{n-1}V\circ f^j, \qquad
V^Y_n=\sum_{j=0}^{n-1}V^Y\circ F^j, \qquad
\tau_n=\sum_{j=0}^{n-1}\tau\circ F^j, \quad
n\ge 1.
\]
Let $G$ be a random variable.
Let $b_n>0$ be a sequence with $b_n\to\infty$,
such that $\inf_{n\ge1} b_n/b_{[\bar\tau^{-1}n+cb_n]}>0$ for each $c>0$.
Assume that \mbox{$b_n^{-1}(\tau_n-n\bar\tau)\to_p0$} as $n\to\infty$.
Then the following are equivalent:
\begin{itemize}
\item[(a)] $b_n^{-1}V_n\to_d G$ on $(X,\mu_X)$ as $n\to\infty$.
\item[(b)] $b_n^{-1}V^Y_{[n/\bar\tau]}\to_d G$ on $(Y,\mu_Y)$ as $n\to\infty$.
\end{itemize}
\end{thm}

\begin{rmk} \label{rmk:induce}
It is a special case of~\cite[Theorem~A.1]{Gouezel07} that (b) implies (a).  
Moreover, instead of condition $b_n^{-1}(\tau_n-n\bar\tau)\to_p0$ it suffices that $b_n^{-1}(\tau_n-n\bar\tau)$ is tight in~\cite[Theorem~A.1]{Gouezel07}.
\end{rmk}

\begin{proof}
Note that
\begin{align*}
\int_Y|V^Y|\,d\mu_Y\le \int_Y\sum_{\ell=0}^{\tau-1}|V\circ f^\ell|\,d\mu_Y
& =\int_Y\sum_{\ell=0}^{\tau(y)-1}|V\circ\pi(y,\ell)|\,d\mu_Y(y)
\\ & = \bar\tau\int_\Delta |V|\circ\pi \,d\mu_\Delta
 = \bar\tau\int_X |V|\,d\mu_X<\infty.
\end{align*}
So $V^Y\in L^1(Y)$ and similarly $\int_Y V^Y\,d\mu_Y=0$.

Define $\hV=V\circ\pi:\Delta\to\R$ and
$\hV_n=\sum_{j=0}^{n-1}\hV\circ \hf^j$.
Since $\pi$ is a measure-preserving semiconjugacy, condition~(a) is equivalent to 
\begin{itemize}
\item[(a$'$)] $b_n^{-1}\hV_n\to_d G$ on $(\Delta,\mu_\Delta)$ as $n\to\infty$.
\end{itemize}

Note that $\mu_Y$ can be viewed as a probability measure on $\Delta$ supported on $Y$.  As such, $\mu_Y\ll \mu_\Delta$.
By Remark~\ref{rmk:SDC},
we obtain that condition~(a$'$) is equivalent to
\begin{itemize}
\item[(a$''$)] $b_n^{-1}\hV_n\to_d G$ on $(Y,\mu_Y)$ as $n\to\infty$.
\end{itemize}

The lap number $N_n:Y\to\Z^+$ is defined by the relation
\[
\tau_{N_n(y)}(y)\le n < \tau_{N_n(y)+1}(y).
\]
For initial conditions $y\in Y$, we write
\[
\hV_n(y)=V^Y_{N_n(y)}(y)+H(\hf^ny),
\]
where $H:\Delta\to\R$ is given by $H(y,\ell)=\sum_{\ell'=0}^{\ell-1}\hV(y,\ell')$.
Now
\begin{align*}
& \mu_Y(y\in Y:b_n^{-1}|H(\hf^ny)|\ge a)
  =\bar\tau \mu_\Delta(y\in Y:b_n^{-1}|H(\hf^ny)|\ge a)
\\ & \qquad \le \bar\tau \mu_\Delta(x\in \Delta:b_n^{-1}|H(\hf^nx)|\ge a)
 = \bar\tau \mu_\Delta(x\in \Delta:b_n^{-1}|H(x)|\ge a)\to0
\end{align*}
as $n\to\infty$ since $H$ is measurable.  Hence condition~(a$''$) is equivalent to
\begin{itemize}
\item[(a$'''$)] $b_n^{-1}V^Y_{N_n}\to_d G$ on $(Y,\mu_Y)$ as $n\to\infty$.
\end{itemize}
It remains to prove that conditions (a$'''$) and (b) are equivalent.
In other words, we must show that
$b_n^{-1}(V^Y_{N_n}-V^Y_{[n/\bar\tau]})\to_p0$ on $(Y,\mu_Y)$.

We recall some properties of the lap number.
By the ergodic theorem, 
$\lim_{n\to\infty}n^{-1}N_n= \bar\tau^{-1}$ a.e.
Also, $\tau_k\le n$ if and only if $N_n\ge k$.  Let $c>0$ and
set $k=k(n)=[n/\bar\tau+cb_n]$.  A calculation 
shows that if $b_n^{-1}|N_n-n/\bar\tau|>c$ then
$b_n^{-1}|\tau_k-k\bar\tau|\ge c\bar\tau+O(b_n^{-1})$,
so 
$b_k^{-1}|\tau_k-k\bar\tau|\ge cb_k^{-1}b_n\bar\tau+O(b_k^{-1})$.
It follows from the assumptions on $\tau_n$ and $b_n$ that
\[
\mu_Y\big(b_k^{-1}|\tau_k-k\bar\tau|\ge cb_k^{-1}b_n\bar\tau+O(b_k^{-1})\big)\to0 \quad\text{as $n\to\infty$},
\]
and hence that
\begin{align} \label{eq:lap}
b_n^{-1}(N_n-[n/\bar\tau])\to_p 0\quad\text{as $n\to\infty$.}
\end{align} 
 
Passing to the natural extension, we can suppose without loss that $F$ is invertible.  For $n\le-1$, we write $V^Y_n=\sum_{j=n}^{-1}V^Y\circ F^j$.
Then
\[
V^Y_{N_n(y)}(y)-V^Y_{[n/\bar\tau]}(y)=
V^Y_{\tN_n(y)}(F^{[n/\bar\tau]}y)
\quad\text{where}\quad 
\tN_n(y)=N_n(y)-[n/\bar\tau].
\]
Since $F$ is measure-preserving, it suffices to show that
$b_n^{-1}V^Y_{\tN_n}\to_p0$.

By the ergodic theorem,
$n^{-1}V^Y_n\to 0$  a.e.\ and hence in probability as $n\to\pm\infty$.
Let $\eps>0$.  We can choose $\tY\subset Y$ with $\mu_Y(\tY)>1-\eps$
and $N_0\ge1$ such that
$|n^{-1}V^Y_n|<\eps$ on $\tY$ for all $|n|\ge N_0$.

For each $n\ge1$, define 
\[
Y_n'=\{y\in Y:|\tN_n(y)|\le N_0\}, \qquad
Y_n''=\{y\in Y:|\tN_n(y)|> N_0\}.
\]
For $y\in Y_n'$, we have
$|V^Y_{\tN_n(y)}|\le \Psi$, where
$\Psi(y)=\sum_{j=-N_0}^{N_0-1}|V^Y(F^jy)|$.
Note that $|\Psi|_1\le 2N_0|V^Y|_1<\infty$,
so 
\begin{align*}
\mu_Y(y\in Y_n':|b_n^{-1}V^Y_{\tN_n(y)}(y)|>\eps)
& \le 
 \mu_Y(b_n^{-1}\Psi>\eps)<\eps
\end{align*}
for $n$ sufficiently large.

For $y\in Y_n''\cap\tY$, we have $\Big|\frac{1}{|\tN_n|}V^Y_{\tN_n}\Big|<\eps$, and hence
$|b_n^{-1}V^Y_{\tN_n(y)}|< \eps b_n^{-1}|\tN_n|$,
so that
\begin{align*}
\mu_Y(y\in Y_n'':|b_n^{-1}V^Y_{\tN_n(y)}(y)|\ge\eps) 
& \le\mu_Y(b_n^{-1}|\tN_n|\ge1) + \eps.
\end{align*}
By~\eqref{eq:lap}, $b_n^{-1}|\tN_n|\to_p0$.  Hence
$b_n^{-1}V^Y_{\tN_n}\to_p0$,  completing the proof.
\end{proof}

\begin{rmk}\label{rmk:A2}
Suppose that $b_n$ is regularly varying of index $1/\alpha$
with  $\alpha>1$.
Then the assumptions on
$b_n$ in Theorem~\ref{thm:induce} are satisfied,
and condition~(b) can be restated as $b_n^{-1}V^Y_n\to_d \bar\tau^{1/\alpha}G$.
\end{rmk}

\section{The Skorohod topologies on $D[0,1]$}
\label{app:topologies}

Let $D[0,1]$ denote the 
c\`adl\`ag space of right-continuous functions $g:[0,1]\to\R$ with left limits.
The uniform topology on $D[0,1]$ is not suitable for many purposes; on the theoretical side it is not separable, and for applications it is too strong since functions must have jumps in exactly the same place in order to be close to each other.

To circumvent these issues, Skorohod~\cite{Skorohod56} introduced four 
topologies on $D[0,1]$ that are separable and sufficiently strong for theoretical purposes, whilst being sufficiently weak to allow the flexibility for functions to be close to each other in reasonable situations.
The four topologies are ordered by
$$
\cJ_1 > \cJ_2 > \cM_2
	\quad\text{and}\quad
	\cJ_1 > \cM_1 > \cM_2
$$
where $>$ means stronger than. The $\cM_1$ and $\cJ_2$ topologies are not comparable.
All these topologies are weaker than the uniform topology.
The $\cJ_2$ topology plays no role in this paper; we define the remaining topologies below.  For simplicity, we restrict to the interval $[0,1]$.  (The differences between $D[0,1]$ and $D[0,\infty)$ are of a purely technical nature.) 
We refer the reader to ~\cite{Skorohod56,Whitt} for more details and proofs.

\paragraph{The Skorohod $\cJ_1$ topology}

The first Skorohod topology, the $\cJ_1$ topology, is metrizable and is defined through the metric $d_{J_1}$ given by
$$
d_{J_1} (g_1,g_2) =\inf_{\lambda \in \Lambda} \big\{   \|g_2\circ \lambda -g_1\| \,\vee\, \|\lambda -id\|  \big\}
$$
for $g_1, g_2\in D[0,1]$, where $\Lambda$ denotes the space of strictly increasing reparametrizations mapping $[0,1]$ onto itself
and $\|\cdot\|$ denotes the uniform norm. This strong topology, which coincides with the uniform topology on the subspace of continuous functions $C[0,1] \subset D[0,1]$, is suitable to define convergence of discontinuous functions when discontinuities and magnitudes of the jumps are close. For instance, if $a_n \to 1$ then the function $g_n=a_n 1_{[\frac12-\frac{1}{n},1]}$
converges to the function $g=1_{[\frac12,1]}$ in the $\cJ_1$ topology as $n\to\infty$.  (Note that $\|g_n-g\|=|a_n|$ for all $n$, so there is no convergence in the uniform topology.)

\paragraph{The Skorohod $\cM_1$ topology}

In many situations, a single jump in the limit function $g$ corresponds to
multiple smaller jumps in the functions $g_n$.  
In this paper, as in~\cite{MZ15}, the jumps of $g_n$ are $o(1)$ and the limit function $g$ has jumps, so a more flexible topology on $D[0,1]$ is required.

The $\cM_1$ topology on $D[0,1]$ is again metrizable and is defined in terms of the Hausdorff distance
between completed graphs of elements of $D[0,1]$.
Given $g\in D[0,1]$, the \emph{completed graph} 
of $g$ is the set 
$$\Gamma(g)=\{(t,s) \in [0,1]\times \R : s=\alpha g(t^-) +(1-\alpha) g(t), \; \alpha \in [0,1] \}.$$
Let $\Lambda^*(g)$ denotes the space of parameterizations $G=(\lambda, \gamma)\colon [0,1] \to \Gamma(g)$
such that $t'<t$ implies either $\lambda(t')<\lambda(t)$, or $\lambda(t')=\lambda(t)$ and $|\gamma(t)-g(\lambda(t))|
\le |\gamma(t')-g(\lambda(t'))|$. Then the $\cM_1$-metric is defined by 
$$
d_{\cM_1} (g_1,g_2) 
	= \inf_{G_i=(\lambda_i, \gamma_i) \in \Lambda^*(g_i)} 
	\big\{   \|\gamma_1 -\gamma_2\| \,\vee\, \|\lambda_1-\lambda_2\|  \big\}
$$
An example in the spirit of Figure~\ref{fig:I}(a) is obtained by
defining $g_n=\frac34 1_{[\frac12-\frac{1}{n},\frac12)}+a_n 1_{[\frac12,1]}$.
If $a_n \to 1$ then $g_n$ converges to
$g=1_{[\frac12,1]}$ in the $\cM_1$ topology as $n\to\infty$,
but not in the $\cJ_1$ topology.

\paragraph{The Skorohod $\cM_2$ topology}

The $\cM_2$ topology on $D[0,1]$ is also defined in terms of the Hausdorff distance
between completed graphs of elements of $D[0,1]$, namely
$d_{\cM_2}(g_1,g_2) = \rho(\Gamma(g_1),\Gamma(g_2))\vee \rho(\Gamma(g_2),\Gamma(g_1))  
$
where
\begin{align*}
\rho(\Gamma(g_1),\Gamma(g_2))
	& =\sup_{(t_1,s_1) \in \Gamma(g_1)} \inf_{(t_2,s_2) \in \Gamma(g_2)} \| (t_1,s_1) - (t_2,s_2) \| .
\end{align*}
(Here, $\| (t_1,s_1) - (t_2,s_2) \|=|t_1-t_2|+|s_1-s_2|$.)
An example in the spirit of Figure~\ref{fig:I}(b) is obtained by
defining 
$g_n=\frac34 1_{[\frac12-\frac{1}{n},\frac12)}+\frac13 1_{[\frac12,\frac12+\frac{1}{n})}
+a_n 1_{[\frac12+\frac{1}{n},1]}$.
If $a_n \to 1$ then $g_n$ converges to
$g=1_{[\frac12,1]}$ in the $\cM_2$ topology as $n\to\infty$,
but not in the $\cM_1$ topology.

An example in the spirit of Figure~\ref{fig:I}(c) is obtained by
defining $g_n=\frac54 1_{[\frac12-\frac{1}{n},\frac12)}+a_n 1_{[\frac12,1]}$,
where $a_n \to 1$.  Then $g_n$ fails to converge in any of the Skorokhod topologies.

We end this appendix with the following instrumental lemma.

\begin{lemma}\label{lem:approx}
Given $g\in D[a,b]$ take  $\bar g\in D[a,b]$ given by $\bar g=1_{[a,b)} g(a) + 1_{\{b\}} g(b)$. Then 
\[
d_{\cM_2, [a,b]}(g,\bar g) \le b-a + A\wedge B,
\]
where 
\begin{align*}
A & = \sup_{t\in [a,b]} (g(a)-g(t)) + \sup_{t\in [a,b]} (g(t)-g(b)), \\
B & = \sup_{t\in [a,b]} (g(t)-g(a)) + \sup_{t\in [a,b]} (g(b)-g(t)).
\end{align*}
\end{lemma}

\begin{proof}
We assume that $g(b)\ge g(a)$
(the case $g(b)<g(a)$ is entirely analogous).
Then $\Gamma(\bar g)=\{(t,g(a)):a\le t\le b\}\cup\{(b,s):g(a)\le s\le g(b)\}$.  Also $\Gamma(g)\subset[a,b]\times\R$ and intersects every horizontal line between $s=g(a)$ and $s=g(b)$.  

For every $(t,s)\in\Gamma(\bar g)$, there exists $t'\in[a,b]$ such that
$(t',s) \in\Gamma(g)$.
Then $\|(t,s)-(t',s)\|\le b-a$ and hence $\rho(\Gamma(\bar g),\Gamma(g))\le b-a$.

It remains to estimate $\rho(\Gamma(g),\Gamma(\bar g))$.  Let $(t,s)\in\Gamma(g)$.
\begin{itemize}
\item
If $s\in[g(a),g(b)]$, then $(b,s)\in\Gamma(\bar g)$ and
$\|(t,s)-(b,s)\|\le b-a$.
\item If $s<g(a)$, then $(t,g(a))\in\Gamma(\bar g)$ and $g(t)\le s<g(a)$, so
$\|(t,s)-(t,g(a))\|= g(a)-s\le g(a)-g(t)=(g(a)-g(t))\wedge (g(b)-g(t))
\le A\wedge B$.
\item If $s>g(b)$, then $(b,g(b))\in\Gamma(\bar g)$ and there exists $t'\in[a,b]$ such that $g(t')\ge s>g(b)$.  Hence
$\|(t,s)-(b,g(b))\|\le b-a+s-g(b)\le b-a+g(t')-g(b)=b-a+(g(t')-g(b))\wedge (g(t')-g(a))\le b-a+A\wedge B$.
\end{itemize}
In all cases, $\inf_{(\bar t,\bar s)\in\Gamma(\bar g)}\|(t,s)-(\bar t,\bar s)\|\le b-a+A\wedge B$ so 
$\rho(\Gamma(g),\Gamma(\bar g))\le b-a+A\wedge B$
completing the proof.
\end{proof}

\paragraph{Acknowledgements}
The research of IM was supported in part by a
European Advanced Grant {\em StochExtHomog} (ERC AdG 320977)
and by CNPq (Brazil) through PVE grant number 313759/2014-6.
We are grateful to Adam Jakubowski for pointing out reference~\cite{BasrakKrizmanic14}.

\def\polhk#1{\setbox0=\hbox{#1}{\ooalign{\hidewidth
  \lower1.5ex\hbox{`}\hidewidth\crcr\unhbox0}}}

\end{document}